\documentclass[11pt, reqno]{amsart}
\usepackage{times,amsmath,amssymb,amsthm,amscd, color}
\usepackage[all]{xy}
\usepackage{hyperref}
\usepackage{comment}
\newtheorem{thm}{Theorem}[section]
\newtheorem{cor}[thm]{Corollary}
\newtheorem{lem}[thm]{Lemma}
\newtheorem{pro}[thm]{Proposition}

\theoremstyle{definition}

\numberwithin{equation}{section}
\newtheorem*{thma}{Theorem A}
\newtheorem*{thmb}{Theorem B}
\newtheorem*{thmc}{Theorem C}

\newcommand{\Cn}{{\mathbb C}^n}

\newcommand{\incn}{\int_{\Cn}}

\newcommand{\re}{{\text{Re}}\,}

\begin{document}

\title
{Toeplitz operators on Fock-Sobolev type spaces}

\date\today
\thanks{}

\author[H. Cho]{Hong Rae Cho}
\address{Department of Mathematics, Pusan National University,
Busan 609-735, Republic of Korea}
\email{chohr@pusan.ac.kr}

\author[J. Isralowitz]{Joshua Isralowitz}
\address{Department of Mathematics and Statistics, University at Albany, 1400 Washington ave., Albany, NY 12222}
\email {jisralowitz@albany.edu}

\author[J.-C. Joo]{Jae-Cheon Joo}
\address{Department of Mathematics, Pohang University of Science and Technology,
Namgu, Pogang,  Republic of Korea}
\email{joo@postech.ac.kr}

\thanks{The first author was supported by the National Research Foundation of
Korea(NRF) grant
funded by the Korea government(MEST) (NRF-2011-0013740).  Also, part of this work was done while the second author was a postdoctoral researcher at Georg-August Universit\"{a}t G\"{o}ttingen, where he was supported by an
Emmy-Noether grant of Deutsche Forschungsgemeinschaft}

\begin{abstract}
In this paper, we characterize operator-theoretic properties (boundedness,
compactness, and Schatten class membership) of Toeplitz operators with positive measure symbols on weighted Fock-Sobolev spaces of fractional order.
\end{abstract}

\maketitle

\section{Introduction}
The primary aim of this paper is to develop the standard theory regarding
boundedness, compactness, and Schatten class membership of Toeplitz operators with nonnegative
measure symbols on weighted Fock-Sobolev spaces of fractional order.

\medskip

Let $\Cn$ be the $n$-dimensional complex Euclidean space and let $dv$ be the Lebesgue volume measure on $\Cn$
that is
normalized so that
$\incn e^{-|z|^2}\, dv(z)=1$.
For a real number $\alpha$, let
\begin{align*}\label{dvalpha}
dv_\alpha(z)=\frac{dv(z)}{(1+|z|)^\alpha}.
\end{align*}

Then for every $0<p<\infty$, we denote by $L^p_\alpha(\mathbb C^n)$ the space of measurable functions $h$ such that
\begin{align*}
\|h\|_{L^p_\alpha}:=\left(
\incn\left|h(z)e^{-\frac12|z|^2}\right|^p \, dv_\alpha(z)
\right)^{\frac 1p}<\infty.
\end{align*}

\noindent Let $H(\Cn)$ be the set of entire functions on $\Cn$.
Then for a given $0 < p < \infty$,
the \emph{weighted Fock space} $F^p_\alpha$ (with $\|\cdot\|_{F^p_\alpha} := \|\cdot\|_{L^p_\alpha})$  is defined as
\begin{align*}
F^p_\alpha := \left\{ f\in H(\Cn) =
\|f\|_{L^p_\alpha}<\infty \right\}.
\end{align*}

\medskip

If $p=2$, then $F^2_\alpha$ equipped with the natural inner product defined by
\begin{equation} \label{CanInnProd}
\langle f, g\rangle_{L^2_\alpha} = \incn f(w) \overline{g(w)} e^{-|w|^2} dv_\alpha(z)
\end{equation} is a reproducing kernel Hilbert space for every real $\alpha$.
However, it is difficult compute the
reproducing kernel of $F^2_\alpha$ explicitly with respect to the inner product \eqref{CanInnProd}.
To avoid this difficulty, in Section \ref{RepKerSec} we discuss a slightly different inner product $\langle \cdot , \cdot
\rangle_\alpha$ on $F^2 _\alpha $ that generates an equivalent norm and enables us to
compute the reproducing kernel of $F^2 _\alpha$ (with respect to this new inner product.) In particular, if $ d\widetilde{v}_\alpha (z):=\frac{dv(z)}{|z|^\alpha}$ then for $\alpha \leq 0$ we will let
\begin{equation}  \label{posinn}
\langle f, g\rangle_\alpha := \incn f(z) \overline{g(z)} e^{-|z|^2}
d\widetilde{v}_\alpha(z),
\end{equation}
and for $\alpha > 0$ we let
\begin{equation} \label{neginn}
 \langle f, g\rangle_\alpha
 := \incn f^-_{\frac\alpha 2}(z) \overline{g^-_{\frac\alpha 2}(z)} e^{-|z|^2} dv(z)
+ \incn f^+_{\frac\alpha 2}(z) \overline{g^+_{\frac\alpha 2}(z)} e^{-|z|^2}
d\widetilde{v}_\alpha(z) \end{equation} where $f^-_{\frac\alpha 2}$ is the Taylor expansion of $f$ up to order  $\alpha/2$ and $f^+ _{\frac\alpha 2} = f - f^-_{\frac\alpha 2}$.

Note that the spaces $F^2 _\alpha$ are in fact very natural generalizations of the so called ``weighted Fock-Sobolev spaces" from \cite{CCK2}.  In particular, given real numbers $\alpha$, $s$ and $0<p<\infty$, the weighted Fock-Sobolev
space $F^p_{\alpha,s}$ is defined as the space of entire functions $f$ such that
$\|\mathcal R^s f\|_{L^p_\alpha}<\infty$, where $\mathcal{R}^s$ is the radial
fractional derivative operator of order $s$ (see \cite{CCK2}).
It was then proved in \cite{CCK2} that the Fock-Sobolev space
$F^p_{\alpha,s}$ coincides with the weighted Fock space
$F^p_{\alpha-sp}$.

Now let $K^\alpha(z,w)$ denote the reproducing kernel for $F^2_\alpha$ with respect
to $\langle \cdot , \cdot \rangle_\alpha$.
If $\mu$ is a complex Borel measure on $\Cn$, then we will define the \emph{Toeplitz operator}
$T^\alpha_\mu$ by the formula \begin{equation} T_\mu ^\alpha f (z) := \incn f(w) K^\alpha (z, w) e^{-|w|^2}
 |w| ^{-\alpha} \, d\mu(w) \label{ToepOpDefNeg} \end{equation} if $\alpha \leq 0$ and \begin{align} T_\mu ^\alpha f (z) := \incn & f_{\frac{\alpha}{2}} ^- (w) \overline{(K_z ^\alpha)^- _{\frac{\alpha}{2}} (w)} e^{-|w|^2} \, d\mu(w)  \label{ToepOpDefPos}  \\ & + \incn f_\frac{\alpha}{2} ^+ (w) \overline{(K_z ^\alpha)^+ _{\frac{\alpha}{2}} (w)} \frac{e^{- |w|^2}}{|w|^\alpha} \, d\mu(w)  \nonumber   \end{align} if $\alpha > 0$ where here $K_z ^\alpha (w) = K_z ^\alpha (w, z)$.  Note that if $\mu$ satisfies \begin{equation}\label{Condition M}
\incn |K^\alpha(z,w)|^2 \frac{e^{-|w|^2}}{(1 + |w|)^\alpha} \,  d|\mu|(w)< \infty
\end{equation} for every $z\in \Cn$,
then the Toeplitz operator $T^\alpha_\mu$ is densely defined on
$F^2_\alpha$.  If $\mu$ satisfies condition \eqref{Condition M} and $\tilde \mu$ is the \emph{Berezin transform}
of $\mu$ (see Section \ref{BCToepOp} for the precise definition), then the following are the main results of this paper:

\begin{thma}\label{t:thma}
Let $\mu$ be a nonnegative Borel measure on $\Cn$ satisfying (\ref{Condition M}) and let $\alpha$ be a real number. Then the following are equivalent for every $1 \leq p < \infty$.
\begin{itemize}
\item[(a)] The Toeplitz operator $T^\alpha_\mu$ extends to a bounded operator on
$F^p_\alpha$.
\item[(b)] The Berezin transform $\tilde\mu$ is bounded on $\Cn$.
\item[(c)] The function $z \mapsto \mu (B(z,r)) $ is bounded for
every $r > 0$.
\end{itemize}
\end{thma}

\begin{thmb}\label{t:thmb}
Let $\mu$ be a nonnegative Borel measure on $\Cn$ satisfying (\ref{Condition M}) and let $\alpha$ be a real number. Then the following are equivalent for every $1 \leq p < \infty$.\begin{itemize}
\item[(a)] $T^\alpha_\mu$ extends to a compact operator on $F^p_\alpha$.
\item[(b)] $\tilde\mu$ is bounded and vanishes at infinity.
\item[(c)] The function $z \mapsto \mu (B(z,r)) $ is bounded and
vanishes at infinity for every $r > 0$.

\end{itemize}
\end{thmb}

\begin{thmc}\label{t:thmc}
Let $\mu$ be a nonnegative Borel measure on $\Cn$ satisfying (\ref{Condition M}) and let $\alpha$ be a real number. Then the following are equivalent for every $0 < p < \infty$.
\begin{itemize}
\item[(a)] $T^\alpha_\mu$ extends to a Schatten $p$ class operator.
\item[(b)] $\tilde\mu \in L^p (\Cn, dv)$.
\item[(c)] The function $ z \mapsto  \mu(B(z,r)) $ is in
$L^p (\Cn, dv)$ for any $r > 0$.
\item[(d)] Let $r > 0$ and let $\{a_j\} = s \mathbb{Z}^{2n}$, where $0 < s <  r\sqrt{2/n}$ so that $\{B(a_j, r)\}$ covers $\Cn$. Then the sequence
$ \{\mu (B(a_j, r))\} \in \ell^p$.
\end{itemize}
\end{thmc}

Note that Theorems A, B, and C for the ordinary Fock spaces $F^p$ (which is defined as the space $F^p _\alpha$ when $\alpha = 0$) were proved by
Isralowitz and Zhu in \cite{IZ}, and results that are similar to Theorems A, B, and C were proved in \cite{IVW, SV} for a very large class of generalized Fock spaces (see Section $\ref{SectionRAOP}$ for a more detailed comparison between our results and those of \cite{IVW, SV}.)

Furthermore, note that condition (c) in Theorems A, B, and C and condition (d) in Theorem C are independent of $r > 0$ in the sense that if one of these conditions is true for some $r_0 > 0$ then they are true for every $r > 0$ (see \cite{IZ} for a proof of this fact). Finally, note that the proofs of Theorems A, B, and C are primarily based on the (by now standard) techniques in \cite{Zhu1} for handling Toeplitz operators with positive Borel measure symbols.  However, the presence of an awkward inner product and reproducing kernel forces us to make tedious modifications to these techniques.

\medskip

This paper consists of five sections.  In the next section,  we will discuss the explicit formula of the
reproducing kernel for $F^2 _\alpha $ with respect to $\langle \cdot, \cdot
\rangle_\alpha$ in more detail and provide useful estimates for this reproducing kernel.  In Sections \ref{BCToepOp} and \ref{SectionSchatten} we will prove Theorems A, B, and C.  Finally, as mentioned before, the last section will contain a comparison of our results with the results in \cite{IVW,SV} and briefly discuss the problem of extending our results to Toeplitz operators on weighted Fock spaces that are defined in terms of the reproducing kernel of $F_\alpha ^2$ with respect to the inner product \eqref{CanInnProd}.

We will end this introduction with a comment on some notation.  For positive quantities $A$ and $B$ (which may depend on a variety of parameters or variables), we will use the notation $A \lesssim B$ if there exists an unimportant constant $C$ such that $A \leq C B$.  The notation $A \gtrsim B$ and $A \approx B$ will have a similar meaning.

\section{The reproducing kernel of $F_\alpha ^2$ with respect to $\langle \cdot, \cdot
\rangle_\alpha$} \label{RepKerSec}

As was mentioned in the introduction, the inner product $\langle \cdot, \cdot \rangle_\alpha$ generates a
Hilbert space norm on $F^2_\alpha$ that is equivalent to the $F^2_\alpha$ norm.  In particular, if we define $\|\cdot\|_{\widetilde{F_\alpha ^p}}$ on $F_\alpha ^p$ by \begin{equation*} \label{ModLpNorm1} \|f\|_{\widetilde{F_\alpha ^p}}  :=  \left(\incn \left|f^-_{\frac\alpha p}(z) e^{- \frac{1}{2} |z|^2}\right| ^p \,  dv(z)\right)^\frac{1}{p} + \left(\incn \left|f^+_{\frac\alpha p}(z)  e^{-\frac12 |z|^2}\right|^p \,
d\widetilde{v}_\alpha(z)\right)^\frac{1}{p} \end{equation*} when $\alpha > 0$ and \begin{equation*} \label{ModLpNorm2} \|f\|_{\widetilde{F_\alpha ^p}}  :=  \left(\incn \left|f(z) e^{- \frac12 |z|^2}\right| ^p \,  d\widetilde{v}_\alpha (z)\right)^\frac{1}{p} \end{equation*} when $\alpha \leq 0$, then we have that $\|\cdot\|_{\widetilde{F_\alpha ^p}}$  and $\|\cdot\|_{F_\alpha ^p}$ are equivalent norms (quasi-norms when $0 < p < 1$) for all $0 < p < \infty$.  To prove this we first need the following pointwise estimate from  \cite{CCK2}, p. 4 (which in fact will often be used throughout the rest of the paper)
\begin{lem}
\label{mvplem}
Let $p, a, t>0$ and $\alpha$ be real. Then there is a constant $C=C(a,t, \alpha)>0$
such that
\begin{equation*}
|f(z)|^p  e^{-a|z|^2} (1+|z|)^{-\alpha} \le C
\int_{|w-z|<t}|f(w)|^p  e^{-a |w|^2}\, dv_\alpha(w)
\end{equation*}
for every entire function $f$ and every $z\in \Cn$.
\end{lem}

\begin{lem} \label{EqOfLpNorms} If $0 < p < \infty$ and $\alpha$ is any real number, then the norms $\|\cdot\|_{F_\alpha ^p}$ and $\|\cdot\|_{\widetilde{F_\alpha ^p}}$ are equivalent on $F_\alpha ^p$ \end{lem}

\begin{proof} It is elementary to check that $(F_\alpha ^p, \|\cdot\|_{\widetilde{F_\alpha ^p}})$ is a Banach space when $p \geq 1$ and is a quasi-Banach space when $0 < p < 1$.  Thus, by the closed graph theorem (see \cite{K}), it is enough to show that \begin{equation*} (F_\alpha ^p, \|\cdot\|_{\widetilde{F_\alpha ^p}}) = (F_\alpha ^p, \|\cdot\|_{F_\alpha ^p}) \end{equation*} as sets.  Since $f = f^-_{\frac\alpha p} + f^+_{\frac\alpha p}$, we trivially have that $(F_\alpha ^p, \|\cdot\|_{\widetilde{F_\alpha ^p}}) \subseteq  (F_\alpha ^p, \|\cdot\|_{F_\alpha ^p}) $ when $\alpha > 0$ so it remains to show that $(F_\alpha ^p, \|\cdot\|_{\widetilde{F_\alpha ^p}}) \supseteq  (F_\alpha ^p, \|\cdot\|_{F_\alpha ^p})$.  To that end, let $f \in   (F_\alpha ^p, \|\cdot\|_{F_\alpha ^p})$.  Since $f^-_{\frac\alpha p}$ is a polynomial, we only need to show that \begin{align*} \incn  & \left|f^+_{\frac\alpha p}(z)  e^{-\frac12 |z|^2}\right|^p  \,
d\widetilde{v}_\alpha(z) \\ & = \int_{|z| \leq 1} \left|f^+_{\frac\alpha p}(z)  e^{-\frac12 |z|^2}\right|^p  \,
d\widetilde{v}_\alpha(z) + \int_{|z| > 1}  \left|f^+_{\frac\alpha p}(z)  e^{-\frac12 |z|^2}\right|^p  \,
d\widetilde{v}_\alpha(z) < \infty \end{align*} 

Since $f$ is entire, the Cauchy estimates (see p. 101 in \cite{Kr}) immediately tell us that the first term above is finite.  Similarly the Cauchy estimates in conjunction with Lemma \ref{mvplem} give us that $\| f^-_{\frac\alpha p}\|_{F_\alpha ^p} \lesssim \| f\|_{F_\alpha ^p}$.  Thus, since  $|z| ^{-\alpha} \approx (1 + |z|)^{-\alpha}$ if $|z| \geq 1$, writing $f^+_{\frac\alpha p} = f - f^-_{\frac\alpha p}$ gives us that the second term is also finite.

The proof when when $\alpha \leq 0$ is similar to the proof when $\alpha > 0$ and will therefore be omitted.

\end{proof}

Note that (as one would obviously expect) $\langle \cdot, \cdot \rangle_\alpha$ is in fact a bounded sesquilinear form on $F_\alpha ^p \times F_\alpha ^q$ if $1 < p < \infty$ where $q$ is the conjugate exponent of $p$ (and is also bounded on $F_\alpha ^1 \times F^\infty$).  In particular, if $f \in F_\alpha ^p$ and $g \in F_\alpha ^q$ then it is clear that \begin{equation} \label{ProdEqn}
 (f) _\frac{\alpha}{2} ^+ (g)_{\frac{\alpha}{2}} ^+ =
\left\{
\begin{aligned}
& (fg)_{\alpha + 1} ^+
&{\rm if}\quad \alpha \not \in \mathbb{N},\\
&(fg)_{\alpha } ^+
&{\rm if}\quad \alpha  \in \mathbb{N}
\end{aligned}
\right.
\end{equation}
and that a similar equation holds for $(f) _\frac{\alpha}{2} ^- (g)_{\frac{\alpha}{2}} ^- .$ Thus, by the standard Cauchy estimates and Lemma \ref{EqOfLpNorms} for $p = 1$, we have that \begin{equation*} |\langle f, g \rangle_\alpha| \lesssim \int_{\Cn} |f(z)g(z)| e^{-\frac{1}{2}|z|^2} \,dv_\alpha(z) \leq \|f\|_{L_\alpha ^p} \|g\|_{L_\alpha ^q}. \end{equation*}
In virtue of Lemma \ref{EqOfLpNorms}, we will use the notation $\| \cdot\|_{F_\alpha ^2}$ to refer to either the $L_\alpha ^2$ norm on
$F_\alpha ^2$ from the introduction or the norm induced  by $\langle \cdot, \cdot
\rangle_\alpha$.

Note that Lemma \ref{mvplem} tells us that $F_\alpha ^2$ is in fact a reproducing kernel Hilbert space.  As is well known, we have
\begin{align}
\label{onbker}
K^\alpha(z,w) = \sum_{\beta}\phi_\beta(z)\overline{\phi_\beta(w)}
\end{align}
where $\{\phi_\beta\}$ is any orthonormal basis for  $F^2_\alpha$
with respect to $\langle\cdot,\cdot\rangle_\alpha$.
Note that polynomials form a dense subset of  $F^2_\alpha$ (see Proposition $2.3$ in \cite{CCK2}). Also, note that
monomials are mutually orthogonal, which means that $\{\frac{z^\beta}{\sqrt{\langle z^\beta, z^\beta\rangle_\alpha}}\}_\beta$ is
an orthonormal basis for  $F^2_\alpha$.  Equation  $(\ref{onbker})$ and arguments that are identical to the ones in the proof of Theorem $4.5$ in \cite{CCK2} then give us that \begin{equation}
\label{RepKerFormual} K^\alpha(z,w)=
\begin{cases}
\mathcal I^{-\alpha/2}K_w(z) &{\rm if}\quad \alpha\leq 0,\\
\mathcal I^{-\alpha/2}K_w(z)+   (K_w)_{\frac{\alpha}{2}} ^- (z)  &{\rm if}\quad \alpha > 0.\\
\end{cases}
\end{equation}

\noindent Here,  $\mathcal I^s$ is the fractional integration operator defined as:
\begin{equation}
\label{fracint}
\mathcal I^s f(z)=
\left\{
\begin{aligned}
&\sum_{k=0}^\infty \frac{\Gamma(n+k)}{\Gamma(n+s+k)} f_k(z)
&{\rm if}\quad s\ge 0,\\
&\sum_{k> |s|}^\infty \frac{\Gamma(n+k)}{\Gamma(n+s+k)} f_k(z)
&{\rm if}\quad s<0.\\
\end{aligned}
\right..
\end{equation}

\noindent Moreover,  for $s > 0,$  $f^+_s$ is the tail part of the Taylor expansion of $f$
of degree higher than $|s|$ given by
\label{taylor} \begin{equation}
f^+_s(z) = \sum_{k> |s|} f_k(z) \end{equation}
and we let $f^-_s = f-f^+_s$ (see\cite{CCK2} for more information on fractional differentiation and integration).

Now if $\alpha \leq 0$ then $(L_\alpha ^2, \|\cdot\|_{\widetilde{F_\alpha ^2}})$ is a closed subspace of $F_\alpha ^2$ with respect to $\langle \cdot, \cdot \rangle_\alpha$.  In this case, let $P_\alpha$ denote the orthogonal projection, so that \begin{equation*} P_\alpha f (z) = \langle f, K_z ^\alpha \rangle_\alpha \end{equation*} for any $f \in L_\alpha ^2$.  Note that the inner product $\langle\cdot,\cdot\rangle_\alpha$ does not make sense on $L_\alpha ^2$ when $\alpha > 0$. Also note that $T_\mu ^\alpha$ would not even make sense for very natural $\mu$ (for example $\mu$ being point-mass measure at the origin) if we gave $F_\alpha ^2$ the inner product $\langle \cdot, \cdot \rangle_\alpha$ when $\alpha > 0$ and defined $T_\mu ^\alpha$ (in the usual way) in terms of this inner product. Furthermore, if $d\mu = f \, dv$ for a measurable function $f$ on $\Cn$, then note that $T_\mu ^\alpha = P_\alpha M_f$ when $\alpha \leq 0$ where $M_f$ is ``multiplication by $f$.''   On the other hand, even though $\langle \cdot, \cdot \rangle_\alpha $ does not make sense on $L_\alpha ^2$, we will show in Section \ref{BCToepOp} that the definition of $T_\mu ^\alpha$ given by $(\ref{ToepOpDefNeg})$ and $(\ref{ToepOpDefPos})$ defines the ``correct'' sesquilinear form on $F_\alpha ^2$ with respect to $\langle \cdot, \cdot \rangle_\alpha$ when $|\mu|$ is Fock-Carleson  (see Section \ref{BCToepOp} for precise definitions, and note that this sesquilinear form itself is sometimes taken as the definition of a Toeplitz operator, see \cite{BC}, p. 583 for example.)

We will finish this section by proving some useful estimates for the reproducing kernel $K^\alpha(z, w)$.

\begin{lem}
\label{weightedmodified}
If $\alpha$ is any real number, then there is a positive constant
$C=C(\alpha)>0$ such that
\begin{equation*}
|K^\alpha(z,w)| \le C \left\{
\begin{aligned}
& (1+|z| |w|)^{\frac\alpha 2}   e^{\frac12 |z|^2 + \frac12 |w|^2 - \frac18 |z - w|^2}  &{\rm if}\quad\alpha<0, \\
& (1+|z\cdot\overline w|)^{\frac\alpha 2} e^{\frac12 |z|^2 + \frac12 |w|^2 - \frac18 |z - w|^2}
&{\rm if}\quad\alpha>0.
\end{aligned}
\right.
\end{equation*}
\end{lem}

\begin{proof}

It was proved in \cite{CCK2},  p. 15 that \begin{equation*}
|\mathcal I^s K_w(z) | \le C \left\{
\begin{aligned}
& (1+|z| |w|)^{-s} \left(|e^{z \cdot \overline{w}}| + e^{\frac12 |z||w|}\right)  &{\rm if}\quad s\geq0, \\
& (1+|z\cdot\overline w|)^{-s}  \left(|e^{z \cdot \overline{w}}| + e^{\frac12 |z||w|}\right)  &{\rm if}\quad s<0.
\end{aligned}
\right.
\end{equation*}

\noindent The proof now follows from the elementary fact that \begin{equation*} |e^{z \cdot \overline{w}}| + e^{\frac12 |z||w|} \leq 2 e^{\frac12 |z|^2 + \frac12 |w|^2 - \frac18 |z - w|^2} \end{equation*} \end{proof}

We now estimate $K^\alpha (z, z)$ for any $z \in \Cn$ as follows:
\begin{pro}\label{weighteddiagonal}
For a real number $\alpha$,
\begin{align*}
|K^\alpha(z,z)|\approx
(1+|z|)^\alpha e^{|z|^2},\quad z\in\Cn.
\end{align*}
\end{pro}

\begin{proof}
By (\ref{RepKerFormual}) , it is enough to prove that \begin{align*}
|\mathcal I^s K(z,z)|\approx(1+|z|)^{-2s}
e^{|z|^2}, \quad z\in\Cn,
\end{align*}
and
\begin{align*}
|\mathcal I^{-s} K(z,z)|\gtrsim |z|^{2s}e^{|z|^2},
\quad |z|\geq\sigma
\end{align*}
for $s > 0$ where the constant of equivalence depends on $\sigma$.

To that end, let $s>0$ and let $s=m+r$ where $m$ is a nonnegative integer and $0\le r<1$.
Then for an entire function $f$, it was proved in \cite{CCK2} p. 14 that
\begin{align}
&\mathcal I^{s}f(z) = \frac{1}{\Gamma(s)}\int_0^1 t^{n-1}(1-t)^{s-1} f (tz)\, dt \label{FracIntPosForm}\\
&\mathcal I^{-s}  f(z) = \frac{1}{\Gamma(1-r)}
    \int_0^1 (1-t)^{-r} t^s \partial_t^{m+1}[t^{n-r} f^+_s (tz)]\, dt \label{FracIntNegForm} \end{align} for any entire $f$.

Let  $\sigma > 0$.  If $|z| \leq \sigma$, then $(\ref{FracIntPosForm})$ tells us that
\begin{align*}
|\mathcal I^s K(z,z)|
&=\frac 1{\Gamma(s)}\int_0^1 t^{n-1}(1-t)^{s-1}e^{t|z|^2}dt\\
&=\frac 1{\Gamma(s)}\int_0^1 (1-t)^{n-1}t^{s-1}e^{(1-t)|z|^2}dt\\
&=\frac {e^{|z|^2}}{\Gamma(s)}\int_0^1 (1-t)^{n-1}t^{s-1}e^{-t|z|^2}dt\\
&\approx (1 + |z|)^{-2s} e^{|z|^2}
\end{align*} since $|z| \leq \sigma$.  However, if $|z| \geq \sigma$, then
\begin{align*}
|\mathcal I^s K(z,z)|
&=\frac {e^{|z|^2}}{\Gamma(s)}\frac 1{|z|^{2s}}
\int_0^{|z|^2}\left(1-\frac\tau{|z|^2}\right)^{n-1}
e^{-\tau}\tau^{s-1}d\tau\\
&\approx (1 + |z|)^{-2s} e^{|z|^2}.
\end{align*}

To finish the proof we will estimate $\mathcal I^{-s} K(z,z)$ with $0<s=m+r$
where $m$ is a non-negative integer and $0\le r<1$. Let $e_k(t) = \sum_{j = k + 1} ^\infty \frac{t^j}{j!}.$ Then by $(\ref{FracIntNegForm})$ we have
\begin{align*}
\left|\mathcal I^{-s} K(z,z) \right|
&=\frac 1{\Gamma(1-s)}\int_0^1 (1-t)^{-r} t^s
\partial_t^{m+1}[t^{n-r} (K_z)^+_s (tz)]\, dt\\
&=\frac 1{\Gamma(1-s)}\int_0^1 (1-t)^{-r} t^s
\partial_t^{m+1}[t^{n-r} e_{m + 1} (t|z|^2)]\, dt.
\end{align*}
Note that $\partial_t^k e_{m + 1} (t) = e_{m+ 1 -k}(t)$ when $k< m + 1$ and $\partial_t^k e_{m + 1} (t)
=e^t$ when $k\ge m + 1 $. Thus, we have
\begin{align*}
\partial_t^{m+1}[t^{n-r} e_{m + 1} (t|z|^2)]
\geq |z|^{2(m+1)}t^{n-r}e^{t|z|^2}.
\end{align*}
Therefore
\begin{align*}
\left|\mathcal I^{-s} K(z,z) \right|
&\gtrsim |z|^{2(m+1)}\int_0^1(1-t)^{-r}t^{s+n-r}e^{t|z|^2}dt\\
&=|z|^{2(m+1)}e^{|z|^2}\int_0^1 t^{-r}(1-t)^{s+n-r}e^{-t|z|^2}dt\\
&=|z|^{2(m+1)}e^{|z|^2}\int_0^{|z|^2}|z|^{2r-2}
\tau^{-r}\left(1-\frac{\tau}{|z|^2}\right)^{s+n-r}e^{-\tau}d\tau\\
&=|z|^{2s}e^{|z|^2}\int_0^{|z|^2}
\tau^{-r}\left(1-\frac{\tau}{|z|^2}\right)^{s+n-r}e^{-\tau}d\tau\\
&\approx |z|^{2s}e^{|z|^2},\quad |z|\geq\sigma.
\end{align*}

\end{proof}

    Finally in this section we will obtain a lower estimate of $K^\alpha(z, w)$ near the diagonal.

    \begin{pro}\label{t:lowerbound}
If $\alpha$ is a real number then there is an $r>0$ such that
\begin{align*}
|K^\alpha(z,w)|\gtrsim (1+|z|)^\alpha
e^{\frac 12|z|^2+\frac 12|w|^2}
\quad\text{for all}\quad w\in B(z,r).
\end{align*}
\end{pro}

\begin{proof}
It was proved in \cite{Li}, p. $420$ that there exists a constant $C = C(r)$ where \label{l:submean}

\begin{equation} \label{l:submean}
\left|\nabla\left(|f(z)|e^{-\frac 12|z|^2}\right)\right|
\leq C\left(\int_{B(z,r)}\left|f(w)e^{-\frac 12|w|^2}\right|^p
dv(w)\right)^{1/p}
\end{equation}
for any $0 < p < \infty$ and $f$ entire, provided $f(z)\neq0$.

By  (\ref{l:submean}), we have
\begin{align*}
&\left||K^\alpha(z,w)|e^{-\frac 12|w|^2}
-|K^\alpha(z,z)|e^{-\frac 12|z|^2}\right|\\
&\lesssim |z-w|\left(\int_{B(z,2r)}\left|K^\alpha(z,\zeta)
e^{-\frac 12|\zeta|^2}\right|^2dv(\zeta)\right)^{1/2}.
\end{align*}

If $|z|\leq 1$, then $|K^\alpha(z,\zeta)|\lesssim 1$ for any $\zeta\in B(z,r)$. Thus
\begin{align*}
&\left||K^\alpha(z,w)|e^{-\frac 12|w|^2}
-|K^\alpha(z,z)|e^{-\frac 12|z|^2}\right|\\
&\lesssim r\left(\int_{B(z,2r)}\left|K^\alpha(z,\zeta)
e^{-\frac 12|\zeta|^2}\right|^2dv(\zeta)\right)^{1/2}\\
&\lesssim r, \quad w\in B(z,r)\quad \text{for small}\quad r>0.
\end{align*}
Hence we have
\begin{align*}
|K^\alpha(z,w)|e^{-\frac 12|w|^2}
&\gtrsim|K^\alpha(z,z)|e^{-\frac 12|z|^2}-r\\
&\gtrsim 1, \quad \text{for small}\quad r>0.
\end{align*}

Now let $|z|\geq 1$.
By  Lemma \ref{weightedmodified},
$|K^\alpha(z,\zeta)|\lesssim (1+|z||\zeta|)^{\alpha/2} E(z, w)$ where \begin{equation*} E(z, w) =
e^{\frac12 |z|^2 + \frac12 |w|^2 - \frac18 |z - w|^2}. \end{equation*}
Also by (\ref{l:submean}), we have
\begin{align*}
&\left||K^\alpha(z,w)|e^{-\frac 12|w|^2}
-|K^\alpha(z,z)|e^{-\frac 12|z|^2}\right|\\
&\lesssim r\left(\int_{B(z,2r)}
\left|K^\alpha(z,\zeta)e^{-\frac 12|\zeta|^2}\right|^2
dv(\zeta)\right)^{1/2}\\
&\lesssim r\left(\int_{B(z,2r)}(1+|z||\zeta|)^{\alpha}
|E(z,\zeta)|^2
e^{-|\zeta|^2}dv(\zeta)\right)^{1/2}.
\end{align*}
Note that if $r\leq 1/2$, then
$|\zeta|\approx |z|$ for $\zeta\in B(z,r)$.
Hence,  we have
\begin{align*}
&\left||K^\alpha(z,w)|e^{-\frac 12|w|^2}
-|K^\alpha(z,z)|e^{-\frac 12|z|^2}\right|\\
&\lesssim r(1+|z|)^\alpha
\left(\incn |E(z,\zeta)|^2
e^{-|\zeta|^2}dv(\zeta)\right)^{1/2}\\
&\lesssim r(1+|z|)^\alpha e^{\frac 12|z|^2}.
\end{align*}

\noindent Finally, by Proposition \ref{weighteddiagonal}, we have
\begin{align*}
|K^\alpha(z,w)|e^{-\frac 12|w|^2}
&\gtrsim|K^\alpha(z,z)|e^{-\frac 12|z|^2}
-r(1+|z|)^\alpha e^{\frac 12|z|^2}\\
&\gtrsim(1-r)(1+|z|)^\alpha e^{\frac 12|z|^2},
\quad w\in B(z,r)\quad\text{for small}\quad r>0.
\end{align*}
\end{proof}

\section{Boundedness and compactness for Toeplitz operators} \label{BCToepOp}

In this section we prove Theorems A and B. As usual, the proofs rely heavily on a natural characterization of Carleson measures.  First we prove some preliminary results. Let $\mu$ be a complex Borel measure in the sense that $\mu$ can be written as $\mu = (\mu_1 - \mu_2) + i (\mu_3 - \mu_4)$ where each $\mu_j$ for $j = 1, \ldots, 4$ is a $\sigma-$finite positive Borel measure on $\Cn$  If $\mu$ further satisfies $(\ref{Condition M})$ then define the Berezin transform $\tilde{\mu}$ of a Borel measure $\mu$ by \begin{equation*}  \tilde{\mu} (z) = \incn |k_z ^\alpha (w) | ^2 \frac{e^{-|w|^2}}{
|w|^\alpha} \, d\mu(w) \end{equation*} if $\alpha \leq 0$ and \begin{equation*} \tilde{\mu} (z) = \incn |(k_z ^\alpha ) ^- _{\frac\alpha 2} (w) | ^2 e^{- |w|^2} \,  d\mu(w) + \incn |(k_z ^\alpha ) ^+  _{\frac\alpha 2} (w) |^2  \,
\frac{e^{-|w|^2}}{|w|^\alpha} \, d\mu(w) \end{equation*} if $\alpha > 0$.  Note that we will prove in this section that (as one would expect)\begin{equation*} \langle T_\mu ^\alpha k_z ^\alpha, k_z ^\alpha \rangle_\alpha = \tilde{\mu} (z) \end{equation*} whenever $T_{|\mu|} ^\alpha$ is bounded.

\begin{lem}\label{l:berezin}
$\mu(B(z,r))\lesssim \tilde\mu(z)$ for small enough $r > 0$ and $|z| \geq 2r$ .
\end{lem}

\begin{proof} First assume that $\alpha \leq 0$.
By Propositions \ref{weighteddiagonal} and
 \ref{t:lowerbound}, we have
\begin{align*}
\tilde\mu(z)&=\incn |k_z^\alpha(w)|^2 \frac{e^{-|w|^2}}{|w|^\alpha} \, d\mu(w)\\
&=\incn\frac{|K^\alpha(z,w)|^2}{K^\alpha(z,z)}\frac{e^{-|w|^2}}{|w|^\alpha}d\mu(w)\\
&\gtrsim\int_{B(z,r)}\frac{(1+|z|)^{2\alpha}
e^{|z|^2+|w|^2}}{(1+|z|)^\alpha e^{|z|^2}}\frac{e^{-|w|^2}}{|w|^\alpha}d\mu(w)\\
&\gtrsim\mu(B(z,r))
\end{align*} since $|z| \geq 3r$.

If $\alpha  > 0$ then \begin{align*}  \tilde{\mu} (z) & = \incn |(k_z ^\alpha ) ^- _{\frac\alpha 2} (w) | ^2 e^{- |w|^2} \,  d\mu(w) + \incn |(k_z ^\alpha ) ^+  _{\frac\alpha 2} (w) |^2  \,
\frac{e^{-|w|^2}}{|w|^\alpha} \, d\mu(w) \\ & \gtrsim \int_{|w| \geq r }  |(k_z ^\alpha ) ^- _{\frac\alpha 2} (w) | ^2 \frac{e^{- |w|^2}}{|w|^\alpha}  \,  d\mu(w) + \int_{|w| \geq r}  |(k_z ^\alpha ) ^+  _{\frac\alpha 2} (w) |^2  \,
\frac{e^{-|w|^2}}{|w|^\alpha} \, d\mu(w) \\ & \gtrsim \int_{|w| \geq r} |k_z^\alpha(w)|^2 \frac{e^{-|w|^2}}{|w|^\alpha} \, d\mu(w). \end{align*} The case $\alpha > 0$ is now identical to the case $\alpha \leq 0$ under the assumption that $|z| \geq 2r$ since then $B(z, r) \subseteq \{w  : |w| \geq r\}$.
\end{proof}

Note that by \eqref{Condition M} we have that \begin{equation*} \sup_{z \in B(0, 2r)} \mu(B(z, r)) < \infty \end{equation*} for any $r > 0$ (where the supremum obviously depends on $r > 0$).

\begin{pro}
\label{ksnorm}
If $0<p<\infty$ and $\alpha$ is real, then
\begin{align*}
\|k^\alpha_z\|_{F^p_\alpha}
\approx (1+|z|)^{\left(\frac 12-\frac 1p\right)\alpha}
\end{align*}
for every $z \in \Cn$.
\end{pro}

\begin{proof}
By Lemma \ref{weightedmodified}, we have
 \begin{align*}
\|K^\alpha(\,\cdot\,,z)\|_{F^p_\alpha}^p
&=\incn\left|K^\alpha(w,z)
 e^{-\frac12|w|^2}\right|^p\,dv_\alpha(w)\\
&\lesssim \incn \left[(1+|z| |w|)^{\frac\alpha 2}E(z,w)
\right]^p e^{-\frac p2|w|^2}\,dv_\alpha(w) \\
&\lesssim \incn\left[(1+|z|)^{\frac\alpha 2}(1+|w|)^{\frac\alpha 2}E(w,z)
\right]^p e^{-\frac p2|w|^2}\,dv_\alpha(w).
\end{align*} where as before \begin{equation*} E(z, w) =
e^{\frac12 |z|^2 + \frac12 |w|^2 - \frac18 |z - w|^2}. \end{equation*}

Therefore, since \begin{equation*} \frac{1 + |z|}{1 + |w|} \leq 1 + |z - w|, \end{equation*} there is a positive constant $C=C(p,\alpha)$ such that
$$
\|K^\alpha(\,\cdot\,,z)\|_{F^p_\alpha}
\le C (1+|z|)^{\alpha-\frac\alpha p} e^{\frac {|z|^2}{2}}.
$$

Combining this with Proposition \ref{weighteddiagonal}, we see that
\begin{align*} \|k^\alpha_z\|_{F^p_\alpha}
\lesssim (1+|z|)^{\left(\frac 12-\frac 1p\right)\alpha}. \end{align*}

To get the lower estimate, fix $r>0$ small enough that Proposition
\ref{t:lowerbound} holds. Then
\begin{eqnarray*}
\|k^\alpha_z\|^p_{F^p_\alpha} &=& \incn |k^\alpha_z(w)|^p e^{-\frac p2|w|^2}\,dv_\alpha(w)\\
&\geq& \int_{B(z,r)} |k^\alpha_z(w)|^p e^{-\frac p2|w|^2}\,dv_\alpha(w)\\
&=& \int_{B(z,r)} \frac{|K^\alpha(z,w)|^p}{|K^\alpha(z,z)|^{\frac p2}} e^{-\frac p2|w|^2}
(1+|w|)^{-\alpha} \,dv(w)\\
&\gtrsim& (1+|z|)^{\left(\frac p2 -1\right)\alpha}
\end{eqnarray*} which completes the proof.
\end{proof}

\begin{def}\label{Carleson} Let $1\leq p <\infty$. A nonnegative Borel measure $\mu$
on $\Cn$ is called a \emph{Carleson measure} for $F^p_\alpha$ if \begin{equation*} \incn \left|f^-_{\frac\alpha p}(z) e^{- \frac12 |z|^2}\right| ^p \,  d\mu(z) + \incn \left|f^+_{\frac\alpha p}(z) e^{-\frac{1}{2} |z|^2}\right|^p   |z|^{- \alpha} \, d\mu(z) \lesssim  \|f\|_{F_\alpha ^p} ^p \end{equation*} when $\alpha > 0$ and  \begin{equation*} \incn \left|f(z) e^{- \frac{1}{2} |z|^2}\right|^p |z|^{-\alpha} \,  d\mu(z)   \lesssim  \|f\|_{F_\alpha ^p} ^p \end{equation*} when $\alpha \leq 0$.  Note that our definition of a Carleson measure for $F_\alpha ^p$ is fundamentally different than than the definition of a Carleson measure for $F_\alpha ^p$ in \cite{CCK2}.  However, our definition is far more convenient for proving results about Toeplitz operators on $F_\alpha ^p$.  Furthermore, as the next theorem shows, our definition is equivalent (modulo the weight factor $(1 + |z|)^\alpha$) to the definition in \cite{CCK2}.
\end{def}

\begin{thm}\label{t:Carleson}
If $\mu$ is a nonnegative Borel measure on $\Cn$ and $\alpha$ is a real number, then the following are equivalent for any $1 \leq p < \infty$.
\begin{itemize}
\item[(a)] $\mu$ is a Carleson measure for $F^p_\alpha$.

\item[(b)] There exists $r > 0$ such that the function $z \mapsto  \mu (B(z,r)) $ is bounded.
\end{itemize}
\end{thm}

\begin{proof}

(a) $\Rightarrow$ (b) : Assume that $r > 0$.  By \eqref{Condition M}, it is enough to consider $|z| \geq 2r$. We only prove that (a) $\Rightarrow$ (b) when $\alpha > 0$ since the proof when $\alpha \leq 0$ is similar. By Proposition \ref{ksnorm}, $k^\alpha_z \in F^p_\alpha$ and
$\|k^\alpha_z\|_{F^p_\alpha} \lesssim (1+|z|)^{\left(\frac12 -
\frac1p\right)\alpha}$. Therefore, since $f = f_\frac{\alpha}{p} ^- + f_\frac{\alpha}{p} ^+$ and $\mu$ is a Carleson measure for $F_\alpha ^p$, we have
\begin{align}\label{e:Carleson}
\int_{B(z,r)} \left|k^\alpha_z(w) e^{- \frac{1}{2} |w|^2}\right|^p   |w|^{-\alpha} \, d\mu(w)
&\lesssim \incn \left|k^\alpha_z(w) e^{-\frac 12|w|^2} \right|^p dv_\alpha(w) \nonumber \\
&\lesssim (1+|z|)^{\left(\frac{p}{2} - 1\right)\alpha}
\end{align}
On the other hand,
\begin{equation} \label{e:normalized kernel}
|k^\alpha_z(w)| \gtrsim (1+|z|)^{\frac\alpha 2} e^{\frac 12|w|^2}
\end{equation}
if $r$ is sufficiently small by Propositions \ref{weighteddiagonal} and  \ref{t:lowerbound}. Plugging \eqref{e:normalized kernel} into
\eqref{e:Carleson}, we get the conclusion.

\medskip

(b) $\Rightarrow$ (a) :  Again we only prove that (b) $\Rightarrow$ (a) when $\alpha > 0$ since the case $\alpha \leq 0$ is similar (and in fact simpler.)

 By Lemma \ref{mvplem}, there exists a constant $C=C(r)$ such
that
$$|f(z)e^{- \frac{1}{2} |z|^2} | ^p |z|^{-\alpha} \leq C \int_{B(z,r/2)} \left|f(w) e^{- \frac{1}{2} |w|^2}\right|^p |w|^{-\alpha} \,dv(w)$$ for all $|z| \geq r$ if $f$ is entire.  Since $B(z,r/2) \subset B(w,r)$ for every $w\in B(z, r/2)$, we have
$$\chi_{B(z,r/2)} (w) \leq \chi_{B(w,r)}(z)$$ for every $z,w\in \Cn$, so that
\begin{align*}
\int_{|z| \geq r} \left|f^+_{\frac\alpha p}(z) e^{-\frac{1}{2} |z|^2}\right|^p   & |z|^{- \alpha} \,d\mu(z)
\lesssim \incn \left(\int_{B(z,r/2)} \left|f^+_{\frac\alpha p}(w) e^{- \frac{1}{2} |w|^2}\right|^p |w|^{-\alpha}  \,dv(w) \right) d\mu(z)\\
&= \incn \left(\incn \chi_{B(z, r/2)} (w) \left|f^+_{\frac\alpha p}(w) e^{- \frac{1}{2} |w|^2}\right|^p |w|^{-\alpha}  \,dv(w) \right) d\mu(z)
\\
&\leq \incn \left(\incn \chi_{B(w, r)} (z) \left|f^+_{\frac\alpha p}(w) e^{- \frac{1}{2} |w|^2}\right|^p |w|^{-\alpha}  \,dv(w) \right)
d\mu(z)\\
& = \incn \mu(B(w,r))  \left|f^+_{\frac\alpha p}(w) e^{- \frac{1}{2} |w|^2} \right|^p |w|^{-\alpha} \,dv(w) \\
&\lesssim \incn \left|f^+_{\frac\alpha p}(w)e ^{-\frac 12 |w|^2} \right|^p \,d\tilde{v}_\alpha(w) \\
& \lesssim \|f\|_{F_\alpha ^p} ^p
\end{align*} where the last inequality follows from Lemma \ref{EqOfLpNorms}.

Now, by the Cauchy estimates and Lemma \ref{mvplem}, one can easily prove that \begin{equation*} \sup_{|z| \leq r}  |f^+_{\frac\alpha p}(z)| |z|^{-\frac{\alpha}{p}} \lesssim  \|f\|_{F_\alpha ^p} \end{equation*} for small enough $r > 0$ independent of $f$.  Thus, since condition (b) is independent of $r > 0$, we can assume $r > 0$ is small enough so that
\begin{align*} \int_{|z| < r} \left|f^+_{\frac\alpha p}(z) e^{-\frac{1}{2} |z|^2}\right|^p   |z|^{- \alpha} \,d\mu(z) \lesssim \|f\|_{F_\alpha ^p} \mu(B(0, r)). \end{align*}

Finally, since $\mu$ is a Carleson measure for the classical Fock space $F ^p$ (see \cite{IZ}) and since $\|f^- _{\frac\alpha p}\|_{F_0 ^p} \lesssim \|f \|_{F_\alpha ^p}$ by the Cauchy estimates, we have that \begin{align*}
\incn \left|f^-_{\frac\alpha p}(z) e^{- \frac12 |z|^2}\right| ^p \,  d\mu(z)  \lesssim \incn \left|f^-_{\frac\alpha p}(z) e^{- \frac12 |z|^2}\right| ^p \,  dv(z) \lesssim  \|f \|_{F_\alpha ^p}
\end{align*} which completes the proof.

\end{proof}

By Theorem \ref{t:Carleson}, we see that $\mu$ being Carleson for
$F^p_\alpha$ is independent of $1\leq p<\infty$ and $\alpha \in \mathbb{R}$. With this observation, we will
call a nonnegative Borel measure $\mu$ a \emph{weighted Fock-Carleson measure} if it
is a Carleson measure for $F^p_\alpha$ for some $1\leq p<\infty$.

Now let $F^\infty$ be the space of all entire $f$ such that function $z \mapsto f(z) e^{-
\frac{|z|}{2}}$ is bounded.

\begin{lem}\label{l:BilForm}
If $f$ is a polynomial, $g \in F^\infty$, and $\mu$ is a complex Borel measure on $\Cn$ such that $|\mu|$ is a weighted Fock-Carleson
measure, then \begin{equation*} \langle T_\mu ^\alpha f, g \rangle_\alpha = \incn f(z) \overline{g(z)} \frac{e^{- |z|^2}}{ |z|^\alpha} \, d\mu(z) \end{equation*} if $\alpha \leq 0$ and \begin{equation*}\langle T_\mu ^\alpha f, g \rangle_\alpha =  \incn f^-_{\frac\alpha 2}(z) \overline{g^-_{\frac\alpha 2}(z)} e^{-|z|^2} d\mu(z)
+ \incn f^+_{\frac\alpha 2}(z) \overline{g^+_{\frac\alpha 2}(z)} \frac{e^{-|z|^2}}{|z|^\alpha} \, d\mu(z) \end{equation*} if $\alpha > 0$.
\end{lem}

\begin{proof}
First we consider the case $\alpha \leq 0$.
By definition we have that
\begin{align*}
\langle  T^\alpha_\mu f, g\rangle_\alpha
 = \int_{\Cn} T_\mu ^\alpha f (w) \overline{g (w) }
 e^{-|w|^2} \,d{\widetilde v}_\alpha (w).
\end{align*}
By arguments that are similar to proof of Theorem \ref{t:Carleson} with $p = 1$, we have that
\begin{align*}
& \incn \incn |K^\alpha (w, u)| |f(u)| |g(w)| e^{- |w|^2}
\frac{e^{- |u|^2}}{ |u|^\alpha} \, d|\mu|(u) \, d\widetilde{v} _\alpha (w) \\ & \lesssim  \incn  \incn
|K^\alpha (w, u)||f(u)| |g(w)| e^{- |w|^2} e^{-|u|^2} \, d\widetilde v_\alpha(u) \, d\widetilde
v _\alpha (w)< \infty.
\end{align*}
Thus, an easy application of Fubini's
theorem and the reproducing property proves Lemma \ref{l:BilForm}  if
$\alpha \leq 0$.

Now we will prove the lemma for $\alpha > 0$.
By definition we have that
\begin{align}
\label{ToepTaylorFormula} \langle  T^\alpha_\mu f, g\rangle_\alpha
=&\int_{\Cn} \left(T_\mu ^\alpha f\right)_{\frac{\alpha}{2}} ^- (w)
\overline{g_{\frac{\alpha}{2}}^-(w)} e^{-|w|^2} \, dv (w)\\
& +\int_{\Cn} \left(T_\mu ^\alpha f \right)_{\frac{\alpha}{2}} ^+(w)
\overline{g_{\frac{\alpha}{2}} ^+(w) }
e^{-|w|^2} \, d{\widetilde v}_\alpha (w). \nonumber
\end{align}

    First, by Fubini's theorem and the proof of Theorem \ref{t:Carleson}, we will have that
\begin{align*}
T_\mu ^\alpha f (w)  =&
\sum_{|m| \leq \frac{\alpha}{2}} \frac{w^m}{m!} \incn  \overline{u}^m f^- _{\frac\alpha 2} (u) e^{- |u|^2} \, d\mu(u) \\ & + \sum_{|m| > \frac{\alpha}{2}} \frac{\Gamma(n +
|m|)}{\Gamma(n - \frac{\alpha}{2} + |m|)}  \frac{w^m}{m!} \incn  \overline{u}^m f^+ _{\frac\alpha 2}(u)
 \frac{e^{- |u|^2}}{|u|^\alpha} \, d\mu(u)
\end{align*}
if we can show that
\begin{align*}
\sum_{|m| >\frac{\alpha}{2}} \frac{ |w|^{|m|} \, \Gamma(n + |m|)}{m! \, \Gamma(n -
\frac{\alpha}{2} + |m|)}   \incn  |u|^{|m|} |f ^+ _{\frac\alpha 2} (u)| \frac{e^{- |u|^2}}{|u|^\alpha}  \,  dv(u) <
\infty \end{align*}
for each $w \in \Cn$. This will then tell us that
\begin{equation} \label{ToepTaylorFormula1}
\left(T_\mu ^\alpha f\right)_{\frac{\alpha}{2}} ^-(w) = \incn
\left(K^\alpha (w, u)\right)_\frac{\alpha}{2} ^-  f _\frac{\alpha}{2} ^- (u) e^{- |u|^2}\, d\mu(u).
\end{equation}

and \begin{equation} \label{ToepTaylorFormula2}
\left(T_\mu ^\alpha f\right)_{\frac{\alpha}{2}} ^+(w) = \incn
\left(K^\alpha (w, u)\right)_\frac{\alpha}{2} ^+  f_\frac{\alpha}{2} ^+ (u) \frac{e^{- |u|^2}}{|u|^\alpha} \, d\mu(u).
\end{equation}

Obviously we may assume that $f(u) = u^{m_0}$ for some multiindex $m_0$ with $|m_0| > \frac{\alpha}{2}$. Note that Stirling's formula
tells us that $\frac{\Gamma(n + |m|)}{ \Gamma(n - \frac{\alpha}{2} + |m|)} \lesssim
|m|^\frac{\alpha}{2}$. Also note that  \begin{align*}  \incn  |u|^{|m| + |m_0| - \alpha} e^{-
|u|^2} \, dv  (u) & \lesssim \incn |u|^{|m| + |m_0| - \frac{\alpha}{2}} e^{-
|u|^2} \, dv(u)  \\  &  \lesssim  \Gamma   \left(n + \frac{|m|+ |m_0|}{2} -
\frac{\alpha}{4}\right).   \end{align*}  Combining this with the
elementary inequality $(m!)^{-1} \leq n^{|m|} (|m|!)^{-1}$ we get that
\begin{align*} & \sum_{|m| > \frac{\alpha}{2}} \frac{|w|^{|m|} \, \Gamma(n +
|m|)}{m! \, \Gamma(n - \frac{\alpha}{2} + |m|)}   \incn  |u|^{|m| + |m_0|}  e^{-
|u|^2} \, dv _\alpha (u)  \nonumber \\ &  \lesssim \sum_{|m| > \frac{\alpha}{2}}
 |w|^{|m|} (|m|!)^{-1} n ^{|m|}  |m|^\frac{\alpha}{2}   \Gamma\left(n + \frac{|m| +
|m_0|}{2} - \frac{\alpha}{4}\right)  < \infty \end{align*} for each $w \in \Cn$ by
another application of Stirling's formula.

Finally, since the proof of Theorem \ref{t:Carleson} tells us that
\begin{align*}
&\incn \incn \left| \left(K^\alpha (w, u)\right)_\frac{\alpha}{2} ^- \right|
\left|f_\frac{\alpha}{2} ^- (u)\right|  \left|g_{\frac{\alpha}{2}} ^-(w) \right| e^{-|w|^2}
e^{- |u|^2} \, d\mu(u) \, dv(w) \\ & \lesssim \incn \incn \left| \left(K^\alpha (w,
u)\right)_\frac{\alpha}{2} ^- \right|  \left|f_\frac{\alpha}{2} ^- (u)\right|
\left|g_{\frac{\alpha}{2}} ^-(w) \right| e^{-|w|^2} e^{- |u|^2} \,
dv(u) \, dv (w) < \infty
\end{align*}
and
\begin{align*} & \incn \incn \left|
\left(K^\alpha (w, u)\right)_\frac{\alpha}{2} ^+ \right|  \left|f_\frac{\alpha}{2} ^+  (u)\right|
\left|g_{\frac{\alpha}{2}} ^+(w)\right| \frac{e^{-|w|^2}}{|w|^\alpha} \frac{e^{- |u|^2}}{|u|^\alpha} \,
d\mu(u) \, dv(w) \\ & \lesssim \incn \incn \left| \left(K^\alpha
(w, u)\right)_\frac{\alpha}{2} ^+ \right|  \left|f_\frac{\alpha}{2} ^+  (u)\right|
\left|g_{\frac{\alpha}{2}} ^+(w)\right| \frac{e^{-|w|^2}}{|w|^\alpha} \frac{e^{- |u|^2}}{|u|^\alpha} \,
dv(u) \, dv (w) < \infty,
\end{align*}
we can plug
$(\ref{ToepTaylorFormula1})$  and $(\ref{ToepTaylorFormula2})$ into $(\ref{ToepTaylorFormula})$  and use Fubini's theorem and the reproducing property to
complete the proof when $\alpha > 0$.
\end{proof}

\begin{cor}\label{c:Berezin} If $\mu$ Fock-Carleson measure, then
\begin{align*}
\tilde\mu(z)=\langle  T^\alpha_\mu k_z^\alpha, k_z^\alpha\rangle_\alpha,
\quad z\in\Cn.
\end{align*}
\end{cor}

\begin{proof} By definition, we have that
\begin{align*} \widetilde{\mu}(z) = \frac{\left[ T_\mu ^\alpha K^\alpha (\cdot, z)
\right] (z) }{K^\alpha (z, z)}. \end{align*} Note that Theorem \ref{t:Carleson}
implies that $T_\mu ^\alpha K^\alpha (\cdot, z) \in L_\alpha ^2$ for each $z
\in \Cn$. Moreover the estimates from the proof of Lemma \ref{l:BilForm} tell us that $T_\mu ^\alpha K^\alpha (\cdot, z)$ is entire for each $z \in
\Cn$.  The reproducing property of $F ^2 _\alpha$ and the definition of $k_z
^\alpha$ immediately completes the proof.
\end{proof}

\begin{thm}\label{t:Boundedness} Let $\mu$ be a nonnegative Borel measure on $\Cn$ and let $\alpha$ be a real number.
Then the following are equivalent for every $1\leq p < \infty$.
\begin{itemize}
\item[(a)] The Toeplitz operator $T^\alpha_\mu$ extends to a bounded operator on $F^p_\alpha$.

\item[(b)] The function $z \mapsto \mu (B(z,r)) $ is bounded for
every $r > 0$.

\item[(c)] The Berezin transform $\tilde\mu$ is bounded on $\Cn$.
\end{itemize}
\end{thm}

\begin{proof}
First, recall from the introduction that (b) is known to be independent of $r > 0$ (again see \cite{IZ}). Now we will prove Theorem \ref{t:Boundedness} by showing that (a) $\Rightarrow$ (c) $\Rightarrow$
(b) $\Rightarrow$ (a).

(a) $\Rightarrow$ (c) : Suppose $T^\alpha_\mu : F^p_\alpha\rightarrow F^p_\alpha$
extends to a bounded operator for some $1\leq p < \infty$.  Note that $k^\alpha_z
\in F^p_\alpha$ for every $0<p<\infty$
by Proposition \ref{ksnorm}.  Since \begin{align*} \widetilde{\mu}(z) \approx
\frac{\left[ T_\mu ^\alpha K^\alpha (\cdot, z) \right] (z) }{(1 + |z|)^\alpha
e^{|z|^2}}, \end{align*}  Lemma \ref{mvplem}, Proposition $\ref{ksnorm}$, and the boundedness of
$T_\mu ^\alpha $ on $F_\alpha ^p$ gives us that \begin{align} (\widetilde{\mu}(z))^p
& \lesssim (1 + |z|)^{(1 - p)\alpha } e^{- \frac{p}{2} |z|^2} \|T_\mu ^\alpha
K^\alpha (\cdot, z) \|_{F_\alpha ^p} ^p  \nonumber \\ & \leq  (1 + |z|)^{(1 -
p)\alpha } e^{- \frac{p}{2} |z|^2} \|T_\mu ^\alpha \|_{F_\alpha ^p} ^p \| K^\alpha
(\cdot, z) \|_{F_\alpha ^p} ^p   \nonumber \\ & \lesssim   \|T_\mu ^\alpha
\|_{F_\alpha ^p} ^p \nonumber \end{align}
\medskip

(c) $\Rightarrow$ (b) : This follows immediately from Lemma  \ref{l:berezin}.

\medskip

(b) $\Rightarrow$ (a) : For $1 <  p <  \infty $, let $ 1< q < \infty$ be the
conjugate exponent. We will only prove that (b) $\Rightarrow$ (a) for $\alpha > 0$ since the proof when $\alpha \leq 0$ is similar.  By the proof of Theorem $5.5$ in \cite{CCK2}, we have that $(F_\alpha ^p)^* =
F_\alpha ^q$ under the pairing $\langle \cdot, \cdot \rangle_\alpha$.  Note that by Lemma \ref{l:BilForm}, Equation (\ref{ProdEqn}), and Theorem \ref{t:Carleson} for $p = 1$,  we have that \begin{align*}
|\langle T^\alpha_\mu f, g\rangle_\alpha|  &\leq  \incn \left|f^-_{\frac\alpha 2}(z) \overline{g^-_{\frac\alpha 2}(z)} \right| e^{-|z|^2} d\mu(z) + \incn \left| f^+_{\frac\alpha 2}(z) \overline{g^+_{\frac\alpha 2}(z)}\right| \frac{e^{-|z|^2}}{|z|^\alpha} \, d\mu(z)
\\ &\lesssim \|f\|_{F^p_\alpha} \|g\|_{F^q_\alpha}
\end{align*} if $f$ is a polynomial and $g \in F^\infty$. Also note that the estimates from the proof
of Lemma \ref{l:BilForm} allow us to conclude that $T_\mu ^\alpha f$ is entire. Thus,
since polynomials are dense in $F_\alpha ^p$ for all $1 \leq  p < \infty$, we have that $T^\alpha_\mu$ extends to a bounded operator on $F_\alpha ^p$.

Now if $p = 1$ then the same arguments gives us that \begin{equation*} |\langle T_\mu ^\alpha f, g \rangle_\alpha | \lesssim \|f\|_{F_\alpha ^1} \|g\|_{F^\infty} \end{equation*} for $f$ a polynomial and $g \in F^\infty$ which completes the proof since the proof of Theorem $5.7$ in \cite{CCK2} tells us that $(F_\alpha ^1)^* = F^\infty$ under the pairing $\langle \cdot, \cdot \rangle_\alpha$.

\end{proof}

If $\mu$ is a nonnegative Borel measure, then define the norm $\|\cdot\|_{L^p _{\alpha, \mu} \cap H(\Cn)}$ on $L^p _{\alpha, \mu} \cap H(\Cn)$ by \begin{equation*} \label{LpMuNorm1} \|f\|_{L^p _{\alpha, \mu} \cap H(\Cn)} ^p =  \incn \left|f^-_{\frac\alpha p}(z) e^{- \frac{1}{2} |z|^2}\right| ^p \,  d\mu(z) + \incn \left|f^+_{\frac\alpha p}(z)  e^{-\frac12 |z|^2}\right|^p
|z|^{-\alpha} \, d\mu(z) \end{equation*} when $\alpha > 0$ and \begin{equation*} \label{LpMuNorm2} \|f\|_{L^p _{\alpha, \mu} \cap H(\Cn)} ^p =  \incn \left|f(z) e^{- \frac12 |z|^2}\right| ^p   |z|^{-\alpha} \, d\mu (z) \end{equation*} when $\alpha \leq 0$

If $1 \leq p < \infty$ and $\mu$ is a weighted
Fock-Carleson measure, then the inclusion
\begin{equation*} \iota_p : F_\alpha ^p \rightarrow L^p_{\alpha, \mu} \cap H(\Cn) \end{equation*} is bounded. We call $\mu$ a
\emph{vanishing weighted Fock-Carleson measure} for $F^p_\alpha$ if $\iota_p$ is
compact.  Note that as one would expect, a simple argument involving Montell's theorem and the standard Cauchy estimates tells us that $\mu$ is a vanishing weighted Fock-Carleson measure if and only if $\|f_j\|_{L^p_{\alpha, \mu} \cap H(\Cn)} \rightarrow 0$ whenever $\{f_j\}$ is a bounded sequence in $F_\alpha ^p$ where $f_j \rightarrow 0$ uniformly on compact subsets of $\Cn$.


We conclude this section with a characterization of vanishing weighted Fock-Carleson
measures that is similar to Theorem \ref{t:Boundedness}.  First, however, we need to prove some standard results for the $F_\alpha ^p$ setting.

Now if $F^{\infty, 0}$ is the subspace of $F^\infty$ such that $z \mapsto f(z) e^{-
\frac{|z|}{2}}$ vanishes at infinity, then by the proof of Theorem $5.8$ in \cite{CCK2} we have that $F_\alpha ^1 = (F^{\infty, 0})^*$ under the pairing $\langle \cdot , \cdot \rangle_\alpha$.

\begin{lem} \label{l:Weak Compactness}  If $1 \leq p < \infty$ then a sequence $\{f_j\} \subset F_\alpha ^p$ converges weak${}^*$ to $0$ if and only if $\{\|f_j\|_{F_\alpha ^p}\}$ is bounded and $f_j \rightarrow 0$ uniformly on compact subsets of $\Cn$.\end{lem}

\begin{proof} Since $F_\alpha ^1 = (F^{\infty, 0})^*$ and $F_\alpha ^q = (F_\alpha ^p)^*$ if $1 < p < \infty$ under the pairing $\langle \cdot, \cdot \rangle_\alpha$, the result immediately follows from Alaoglu's theorem, Montell's theorem, and the fact that weak${}^*$ convergence implies pointwise convergence. \end{proof}

\begin{lem} \label{l:Compactness Criteria}   If $1 < p < \infty$ and $X$ is a normed space, then $T : F_\alpha ^p \rightarrow X$ is compact if and only if $\{\|f_j\|_{F_\alpha ^p}\}$ is bounded and $f_j \rightarrow 0$ uniformly on compact subsets of $\Cn$ implies that $\|T f_j \|_X \rightarrow 0$.  Furthermore, sufficiently holds when $p = 1$.  \end{lem}

\begin{proof} Sufficiency follows immediately from Alaoglu's theorem.

For necessity, pick a bounded sequence $\{f_j\} \subset F_\alpha ^p$ where $f_j \rightarrow 0$ uniformly on compact subsets of $\Cn$. Since $T$ is compact, passing to a subsequence if necessary, we can assume that $T f_j \rightarrow u$  in $X$.  Then since $f_j \rightarrow 0$ weakly by reflexitivity and Lemma \ref{l:Weak Compactness}, we have for any $v \in X^*$ that \begin{equation*} v(u)  = \lim_{j \rightarrow \infty} v(T f_j)  = \lim_{j \rightarrow \infty} v \circ T (f_j) = 0\end{equation*} which implies that $u = 0$.   \end{proof}

\begin{cor} \label{c: NormRepKerAndCompactness} If $1 \leq p < \infty$ and $T_\mu ^\alpha$ is compact on $F_\alpha ^p$ then \begin{equation*} \|T^\alpha_\mu (k^\alpha_z / \|k^\alpha_z\|_{F^p_\alpha})\|_{F^p_\alpha}
\rightarrow 0 \end{equation*} as $|z|\rightarrow \infty$. \end{cor}

\begin{proof} By the estimates from Lemma \ref{weightedmodified}, Proposition \ref{weighteddiagonal}, and Proposition \ref{ksnorm} we  have that $k_z ^\alpha / \|k_z ^\alpha\|_{F_\alpha ^p}  \rightarrow 0$ uniformly on compact subsets of $\Cn$ as $|z| \rightarrow \infty$. Thus, the Corollary immediately follows from Lemma \ref{l:Compactness Criteria} if $1 < p < \infty$.

On the other hand, if $p = 1$ and $T_\mu ^\alpha$  is compact, then let $|z_j| \rightarrow \infty$ as $j \rightarrow \infty$. Passing to a subsequence if necessary, assume that $T_\mu ^\alpha (k_{z_j} ^\alpha / \|k_{z_j} ^\alpha\|_{F_\alpha ^1}) \rightarrow f$ in $F_\alpha ^1$ for some $f \in F_\alpha ^1$.  However, since the reproducing property holds for $f \in F_\alpha ^1$ and since $\langle \cdot, \cdot \rangle_\alpha$ is bounded on $F_\alpha ^1 \times F^\infty$, we have that \begin{align*} f(w) & = \langle f, K^\alpha (\cdot, w )  \rangle_\alpha \nonumber \\ & = \lim_{j \rightarrow \infty} \langle T_\mu ^\alpha (k_{z_j} ^\alpha / \|k_{z_j} ^\alpha\|_{F_\alpha ^1}), K^\alpha  (\cdot, w )  \rangle_\alpha \nonumber   \end{align*}

However, since $\mu$ is a weighted Fock-Carleson measure, Theorem \ref{t:Carleson} and Lemma \ref{l:BilForm}  gives us that \begin{equation} |\langle T_\mu ^\alpha (k_{z_j} ^\alpha / \|k_{z_j} ^\alpha\|_{F_\alpha ^1} ),  K^\alpha (\cdot, w )  \rangle_\alpha| \lesssim \incn (|k_{z_j} ^\alpha (u)| / \|k_{z_j} ^\alpha\|_{F_\alpha ^1}) |K^\alpha (\cdot, w )| e^{-|u|^2} \, dv_\alpha(u) \label{NormRepKerAndCompactnessEst1}. \end{equation} Moreover, plugging in the estimates from Lemma \ref{weightedmodified}, Proposition \ref{weighteddiagonal}, and Proposition \ref{ksnorm} into (\ref{NormRepKerAndCompactnessEst1}) we get that \begin{equation} |\langle T_\mu ^\alpha (k_{z_j} ^\alpha / \|k_{z_j} ^\alpha\|_{F_\alpha ^1}), K^\alpha ( \cdot, w )  \rangle_\alpha| \lesssim (1 + |w|)^\frac{\alpha}{2} (1 + |z_j|)^\frac{\alpha}{2}  e^{\frac{1}{2} |w|^2} \incn e^{-\frac{1}{8} (|u - z_j|^2 + |w - u|^2)} dv(u). \label{NormRepKerAndCompactnessEst2}  \end{equation}

However, it is very easy to see that the right hand side of (\ref{NormRepKerAndCompactnessEst2}) converges to $0$ (for $w$ fixed) as $j \rightarrow \infty$. In particular, assuming $|w| \leq \frac{1}{4} |z_j|$, if $|u - z_j| \leq \frac{1}{4} |z_j|$ then \begin{equation*} |z_j| \leq |z_j - u| + |u - w| + |w| \leq \frac{1}{2}|z_j| + |u - w|  \end{equation*} so that \begin{equation*} e^{-\frac{|u - w|^2}{8}} \leq e^{- \frac{|z_j|^2}{32}} \end{equation*} which completes the proof.   \end{proof}


\begin{thm}\label{t:Vanishing Carleson} Let $\mu$ be a nonnegative Borel measure on
$\Cn$ and let $\alpha$ be a real number.  Then the following are equivalent for any $1<p<\infty$.

\begin{itemize}
\item[(a)] $T^\alpha_\mu$ extends to a compact operator on $F^p_\alpha$.

\item[(b)] $\tilde\mu$ is bounded and vanishes at infinity.

\item[(c)] The function $z \mapsto \mu (B(z,r))$ is bounded and vanishes at infinity for each $r > 0$.

\item[(d)] $\mu$ is a vanishing weighted Fock-Carleson measure for $F^p_\alpha$.
\end{itemize}

\end{thm}

\begin{proof}  As was mentioned in the introduction, condition (c) is independent of $r > 0$.  We will now prove
Theorem \ref{t:Vanishing Carleson} by showing that (a) $\Rightarrow$ (b) $\Rightarrow$ (c) $\Rightarrow$ (d)
$\Rightarrow$ (a).

(a) $\Rightarrow$ (b) : By proposition \ref{ksnorm}, we have that
\begin{eqnarray*}
\tilde\mu(z) &=& \langle T^\alpha_\mu k^\alpha_z, k^\alpha_z\rangle_\alpha \leq
\|T^\alpha_\mu (k^\alpha_z / \|k^\alpha_z\|_{F^p_\alpha})\|_{F^p_\alpha}
\|k^\alpha_z\|_{F^p_\alpha} \|k^\alpha_z\|_{F^q_\alpha} \\
&\approx& \|T^\alpha_\mu (k^\alpha_z / \|k^\alpha_z\|_{F^p_\alpha})\|_{F^p_\alpha}
(1+|z|)^{\left(\frac12 - \frac1p\right)\alpha} (1+|z|)^{\left(\frac12 -
\frac1q\right)\alpha}\\
 &=& \|T^\alpha_\mu (k^\alpha_z / \|k^\alpha_z\|_{F^p_\alpha})\|_{F^p_\alpha}
\rightarrow 0
\end{eqnarray*} where the last equality follows from Corollary \ref{c: NormRepKerAndCompactness}.

(b) $\Rightarrow$ (c) : This follows from Lemma \ref{l:berezin}.

\medskip

(c) $\Rightarrow$ (d) : We only prove this for $\alpha > 0$ since the proof for $\alpha \leq 0$ is similar (and in fact simpler).  Let $\{f_j\}$ be a bounded sequence in $F^p _\alpha$ such that $f_j
\rightarrow 0$  converges uniformly on every compact
subset of $\Cn$ as $j \rightarrow \infty$.

Now suppose that $\mu$ is a  weighted Fock-Carleson measure which satisfies
$$\mu (B(z,r)) \rightarrow 0$$ as $|z|\rightarrow \infty$ for some
$r>0$. Let $\{f_j\}$ be as above. To complete the proof, it suffices to show
$$\|f_j\|_{L^p_{\alpha, \mu}} \rightarrow 0$$ as $j\rightarrow \infty$. Let
$\varepsilon>0$ be given.  Furthermore, we only show that \begin{eqnarray*}\lim_{j \rightarrow \infty} \incn \left|(f_j)^+_{\frac\alpha p}(z)  e^{-\frac{1}{2} |z|^2}\right|^p |z|^{-\alpha} \, d\mu(z)  = 0 \end{eqnarray*} since the proof that \begin{equation*}  \lim_{j \rightarrow \infty} \incn \left|(f_j)^-_{\frac\alpha p} (z) e^{- \frac12 |z|^2}\right| ^p \,  d\mu(z) = 0 \end{equation*} is similar.

To that end, choose $R >2r$ such that $\mu(B(z, r)) < \varepsilon$ whenever $|z| > R/2$.  Since $\{\|f_j\|_{F_\alpha ^p}\}$ is bounded, we can easily use the Cauchy estimates to estimate the Taylor coefficients of $(f_j)^+_{\frac\alpha p} $ and conclude that $z \mapsto |(f_j)^+_{\frac\alpha p} (z)|^p |z|^{-\alpha}$ is equicontinuous at the origin, so there exists $0 < \delta < 1$ where \begin{equation*} \sup_{j \in \mathbb{N}, |z| < \delta}  |(f_j)^+_{\frac\alpha p} (z)|^p |z|^{-\alpha} < \frac{\varepsilon}{\mu(B(0, 1))}. \end{equation*}
  Moreover, since $(f_j)^+_{\frac\alpha p} \rightarrow 0$ uniformly on compact sets, we have that \begin{equation*} \limsup_{j \rightarrow \infty} \int_{\delta \leq |w| \leq R} |(f_j)^+_{\frac\alpha p} (w)|^p \, \frac{e^{-\frac p2|w|^2}}{ |w|^\alpha} \, d\mu(w) = 0. \end{equation*}
Now choose $s>0$ small enough that $\{B(a_j, r)\}$ covers $\Cn$, where $\{a_j\}$ is an
arrangement of the lattice set $s\mathbb{Z}^{2n}$. We denote by $\{b_j\}$ a
rearrangement of the set $\{a_j : |a_j| > R/2\}$. Then $\{B(b_j,r)\}$ covers $\{z:
|z| >R\}$, so
\begin{align*}
\limsup_{j \rightarrow \infty} \incn \left|(f_j)^+_{\frac\alpha p}(w)\right|^p   \frac{e^{-\frac{p}{2} |w|^2}}{|w|^\alpha} \, d\mu(w)  &= \limsup_{j \rightarrow \infty} \int_{|w|< \delta} |(f_j)^+_{\frac\alpha p}(w)|^p \, \frac{e^{-\frac p2|w|^2}}{|w|^\alpha}\,  d\mu(w) \\ & + \limsup_{j \rightarrow \infty} \int_{\delta \leq |w| \leq R} |(f_j)^+_{\frac\alpha p}(w)|^p \, \frac{e^{-\frac p2|w|^2}}{|w|^\alpha}   \, d\mu(w) \\ & +
\limsup_{j \rightarrow \infty} \int_{|w|>R} |(f_j)^+_{\frac\alpha p}(w)|^p \, \frac{e^{-\frac p2|w|^2}}{|w|^\alpha} \,  d\mu(w)\\
& \leq \varepsilon  + \limsup_{j \rightarrow \infty} \sum_k \int_{B(b_k,r)} |(f_j)^+_{\frac\alpha p}(w)|^p \,\frac{e^{-\frac p2|w|^2}}{|w|^\alpha}\,d\mu(w).
\end{align*}
Since $|w| > 2r$ if $|w| > R$, Lemma \ref{mvplem} tells us that there exists a constant $C$ such that
\begin{equation*}
|(f_j)^+_{\frac\alpha p}(w)|^p \,\frac{e^{-\frac p2|w|^2}}{|w|^\alpha} \leq C \int_{B(b_k, 2r)} |(f_j)^+_{\frac\alpha p}(z)|^p \, \frac{e^{-\frac p2|z|^2}}{|z|^\alpha} \,
dv(z)
\end{equation*}
for every $w\in B(b_k, r)$. Therefore,
\begin{align*}
\incn & \left|(f_j)^+ _{\frac\alpha p}(w)\right|^p   \frac{e^{-\frac{p}{2} |w|^2}}{|w|^\alpha} \, d\mu(w)  \\ & \leq 2\varepsilon +C \sum_k \mu(B(b_k, r))
\int_{B(b_k,2r)}|(f_j)^+ _{\frac\alpha p}(z)|^p \frac{e^{-\frac{p}{2} |z|^2}}{|z|^\alpha} \, dv(z).
\end{align*}
Since $\mu(B(b_k,r)) <\varepsilon$, for a suitable constant $C$, we have
\begin{align*} \limsup_{j \rightarrow \infty} \incn & \left|(f_j)^+ _{\frac\alpha 2}(w)\right|^p   \frac{e^{-\frac{p}{2} |w|^2}}{|w|^\alpha} \, dv(w) \\ & \leq \varepsilon +C \varepsilon \limsup_{j \rightarrow \infty} \sum_k
\int_{B(a_k,2r)}|(f_j)^+_{\frac\alpha p}(z)|^p  \frac{e^{-\frac{p}{2} |z|^2}}{|z|^\alpha} \, d\mu(z). \end{align*}
From the local finiteness of the covering $\{B(a_k, 2r)\}$, there exists a positive
integer $N$ such that
\begin{equation*}
\sum_k  \int_{B(a_k,2r)}|(f_j)^+_{\frac\alpha p} (z)|^p \,\frac{e^{-\frac{p}{2} |z|^2}}{|z|^\alpha} \, dv(z)  \leq N \|f_j\|_{F^p_\alpha}^p \leq N M^p.
\end{equation*}
Therefore, we have
\begin{equation*}   \limsup_{j \rightarrow \infty} \incn \left|(f_j)^+_{\frac\alpha p}(w)\right|^p   \frac{e^{-\frac{p}{2} |w|^2}}{|w|^\alpha} \, d\mu(w)  \leq \varepsilon + \varepsilon  C N M^p\end{equation*} which completes the proof since $\varepsilon$ is arbitrary.

\medskip

(d) $\Rightarrow$ (a) : We only prove this for $\alpha > 0$ since the case $\alpha \leq 0$ is similar. Also for convenience we will let $F_\alpha ^\infty := F^\infty$ for any $\alpha$.

 First we claim that Lemma \ref{l:BilForm} holds for $f \in F_\alpha ^p$ and $g \in F_\alpha ^q$.  To see this, let $\{f_j\}$ be a sequence of polynomials where $f_j \rightarrow f$ in $F_\alpha ^p$ and similarly pick $\{g_j\} \subset F^\infty$ where $g_j \rightarrow g$ in $F_\alpha ^q$.  Since $T_\mu ^\alpha $ is bounded (and thus $\mu$ is a weighted Fock-Carleson measure), we have that \begin{align*} \langle T^\alpha_\mu f, g\rangle_\alpha & = \lim_{j \rightarrow \infty} \langle T^\alpha_\mu f_j, g_j \rangle_\alpha \\ & =    \lim_{j \rightarrow \infty}  \incn (f_j) ^- _{\frac\alpha 2}(z) \overline{(g_j)^-_{\frac\alpha 2}(z)} e^{-|z|^2} d\mu(z)
+ \lim_{j \rightarrow \infty}  \incn (f_j) ^+ _{\frac\alpha 2}(z) \overline{(g_j) ^+_{\frac\alpha 2}(z)} \frac{e^{-|z|^2}}{|z|^\alpha} \, d\mu(z).                                                                                                                                                                                                                                                                                                                                                                                                                                                                                                                          \\ & =   \incn f ^- _{\frac\alpha 2}(z) \overline{g ^-_{\frac\alpha 2}(z)} e^{-|z|^2} \, d\mu(z)
+   \incn f ^+ _{\frac\alpha 2}(z) \overline{g ^+_{\frac\alpha 2}(z)} \frac{e^{-|z|^2}}{|z|^\alpha} \, d\mu(z).\end{align*}

Now let $\{f_j\}$ be a bounded sequence in
$F^p_\alpha$ converging to $0$ uniformly on compact subsets of $\Cn$. By Lemma \ref{l:Compactness Criteria}, it is enough to show that $\lim_{j \rightarrow \infty} \|T^\alpha_\mu f_j\|_{F^p_\alpha} = 0 $.  However, the previous paragraph gives us that
\begin{align*}
\|T^\alpha_\mu f_j\|_{F^p_\alpha} &= \sup\{|\langle T^\alpha_\mu f_j, g\rangle_\alpha|
:  \|g\|_{F^q_\alpha} \leq 1 \}  \\ & \leq  \sup\left\{\left| \incn (f_j) ^-_{\frac\alpha 2}(z) \overline{g^-_{\frac\alpha 2}(z)} e^{-|z|^2}  d\mu(z) \right| : \|g\|_{F^q_\alpha} \leq 1 \right\}  \\ & + \sup\left\{\left| \incn (f_j) ^+_{\frac\alpha 2}(z) \overline{g^+_{\frac\alpha 2}(z)} \frac{e^{-|z|^2}}{|z|^\alpha} \, d\mu(z) \right|:  \|g\|_{F^q_\alpha} \leq 1 \right\}
\\ &  (A_j) + (B_j)
\end{align*}

We only show that $\lim_{j \rightarrow \infty} B_j = 0$ since the proof that $\lim_{j \rightarrow \infty} A_j = 0$ is similar. Let $\varepsilon > 0$. Again by the standard Cauchy estimates, we have that  $z \rightarrow |(f_j) ^+_{\frac\alpha 2} (z) g^+_{\frac\alpha 2} (z)| |z| ^{-\alpha}$ is equicontinuous (with respect to both $j$ and $g$  where  $\|g\|_{F^q_\alpha} \leq 1$) at the origin.  Thus, there exists $\delta > 0$ such that \begin{equation*} \sup\left\{ \int_{|z| \leq \delta} \left| (f_j) ^+_{\frac\alpha 2}(z) g^+_{\frac\alpha 2}(z) \right| \frac{e^{-|z|^2}}{|z|^\alpha} \, d\mu(z) :  \|g\|_{F^q_\alpha} \leq 1 \text{ and } j \geq 1 \right\}  \leq \varepsilon. \end{equation*}

On the other hand, if \begin{equation*} f_j(z) = \sum_{k = 0}^\infty (f_j)_k (z) \end{equation*} is the homogenous expansion of $f_j$, then by the standard Cauchy estimates, H\"{o}lder's inequality, and the fact that $\mu$ is a weighted Fock-Carleson measure, we have that \begin{equation*} \lim_{j \rightarrow \infty} \int_{|z| > \delta}  \left| (f_j)_k (z) g^+_{\frac\alpha 2}(z) \right| \frac{e^{-|z|^2}}{|z|^\alpha} \, d\mu(z) = 0 \end{equation*} for each $k$.  Thus,
if $\|g\|_{F^q_\alpha} \leq 1$, then \begin{align*} \limsup_{j \rightarrow \infty} \int_{|z| > \delta}  \left| (f_j) ^+_{\frac\alpha 2}(z) g^+_{\frac\alpha 2}(z) \right| \frac{e^{-|z|^2}}{|z|^\alpha} \, d\mu(z) & = \limsup_{j \rightarrow \infty} \int_{|z| > \delta}  \left| (f_j) ^+_{\frac\alpha p}(z) g^+_{\frac\alpha 2}(z) \right| \frac{e^{-|z|^2}}{|z|^\alpha} \, d\mu(z) \\ & \lesssim \limsup_{j \rightarrow \infty} \incn  \left| (f_j) ^+_{\frac\alpha p}(z) g^+_{\frac\alpha 2}(z) \right| e^{-|z|^2}{(1 + |z|)^{-\alpha}} \, d\mu(z) \\ & \lesssim \limsup_{j \rightarrow \infty}  \|f_j \|_{L^p _{\mu, \alpha} \cap H(\Cn) } \|g \|_{F_\alpha ^q} = 0 \end{align*}  since $\mu$ is both a Carleson measure for $F^s_\alpha$ for every $ 1 \leq s < \infty$ by Theorem \ref{t:Boundedness} and a vanishing weighted Fock-Carleson measure for $F_\alpha ^p$ which immediately tells us that $\lim_{j \rightarrow \infty} B_j = 0$ and completes the proof.
\end{proof}


\section{Charaterization of Schatten class membership} If $T$ is a positive \label{SectionSchatten}
compact operator on $F^2_\alpha$, then there exists an orthonormal set $\{e_j\}$ in
$F^2_\alpha$ and a non-increasing sequence $\{s_j\}$ of positive real numbers such
that
\begin{align*} T f = \sum_j s_j \langle f, e_j\rangle_\alpha e_j \end{align*} for every $f\in F^2_\alpha$.
The operator $T$
is said to belong to the \emph{Schatten class} $S^\alpha_p$ if the sequence
$\{s_j\}$ of eigenvalues belongs to $\ell^p$. We denote by
$\|T\|_{S^\alpha_p}$ the $\ell^p$-``norm'' $\left(\sum s_j^p\right)^{1/p}$ for
$0<p<\infty$. Note that $\|\cdot \|_{S^\alpha_p}$ is only a quasi-norm when $0 < p <
1$, though for convenience we will still sometimes refer to it as the $S_p ^\alpha$
norm.   If we do not assume that $T$ is positive, then we say that $T$ belongs to
$S^\alpha_p$ whenever $|T| = (T^*T)^{1/2}$ does, and if the $s_j$'s are the eigenvalues
of $|T|$, then we define $\|T\|_{S^\alpha_p} := \left(\sum_j s_j^p\right)^{1/p}$. Note that the space $S^\alpha_1$ is usually called the \emph{trace class}. For a positive operator
$T$, let $\mathrm{tr} (T)$ denote the usual trace defined by
\begin{align*}\mathrm{tr} (T) = \sum_j \langle Te_j, e_j\rangle_\alpha = \|T\|_{S_1 ^\alpha}.\end{align*}

\begin{lem}\label{l:Trace} If $T$ is a positive operator on $F^2_\alpha$, then
\begin{align*} \mathrm{tr}(T) \approx \incn \widetilde T (z)\, dv(z).\end{align*} where \begin{equation*} \widetilde{T} (z) = \langle T k_z ^\alpha, k_z ^\alpha \rangle_\alpha \end{equation*} is the Berezin transform of $T$.
In particular, $T$ is trace-class if and only if the integral above
converges.
\end{lem}

\begin{proof} Let $T = S^2$ for some $S \geq 0$. Let $\{e_j\}$ be an orthonormal
basis for $F^2_\alpha$. Then by Fubini's theorem,
\begin{eqnarray*}
\mathrm{tr}(T)&=& \sum_j \langle Te_j, e_j \rangle_\alpha \approx \sum_j
\|Se_j\|_{F^2_\alpha}^2 = \sum_j \incn |Se_j(z)|^2 e^{-|z|^2} \,dv_\alpha(z)\\
&=& \incn \left(\sum_j|Se_j(z)|^2\right) e^{-|z|^2} \,dv_\alpha(z)\\
&=& \incn \left(\sum_j|\langle Se_j, K^\alpha_z\rangle_\alpha|^2\right) e^{-|z|^2}
\,dv_\alpha(z)\\
&=& \incn \left(\sum_j|\langle e_j, SK^\alpha_z\rangle_\alpha|^2\right) e^{-|z|^2}
\,dv_\alpha(z)\\
&\approx& \incn \|SK^\alpha_z\|^2_{F^2_\alpha} e^{-|z|^2} \,dv_\alpha(z)\\
&\approx&\incn \langle TK^\alpha_z, K^\alpha_z\rangle_\alpha e^{-|z|^2} \,dv_\alpha(z) =
\incn \widetilde T(z) K^\alpha(z,z) e^{-|z|^2} \,dv_\alpha(z)\\
&\approx& \incn \widetilde T(z) \, dv(z),
\end{eqnarray*}
where $K_z^\alpha (w) = K^\alpha(w,z)$ and the last estimate is due to Proposition
\ref{weighteddiagonal}.
\end{proof}

We will need one more result before we prove the main result of this section.  If $\mu$ is a positive Borel measure on $\Cn$ then let $\mu_1$ be defined by \begin{equation*} \mu_1 (E) = \mu(E \cap B(0, 1)) \end{equation*} for any Borel set $E \subset \Cn$ and let $\mu_2$ be defined by \begin{equation*} \mu_2 (E) = \mu(E \cap (\Cn \backslash B(0, 1)))\end{equation*} for any Borel set $E \subset \Cn$

\begin{lem} \label{l:trunToep}  If $\mu$ satisfies condition $(\ref{Condition M})$ then $T_{\mu_1} ^\alpha \in S_p ^\alpha$  for each $p \geq 1$.  \end{lem}

\begin{proof} We only prove this for $\alpha > 0$ since the proof for $\alpha \leq 0$ is similar.  Obviously it is enough to show that $T_{\mu_1} ^\alpha $ is trace class.  For each multiindex $\beta$, let $e_\beta$ be the standard normalized (with respect to $\langle \cdot, \cdot \rangle_\alpha$) monomial \begin{equation*}
e_\beta (z) = \begin{cases}
\sqrt{\frac{\Gamma (n + |\beta|)}{\beta !\Gamma (n + |\beta| - \frac{\alpha}{2})}} z^\beta & \text{ if } |\beta| > \frac{\alpha}{2} \\  (\beta!)^{-\frac{1}{2}} z^\beta & \text{ otherwise.} \end{cases} \end{equation*}

However, by Lemma \ref{l:BilForm}, we have that \begin{equation*} \langle T_{\mu_1} ^\alpha e_\beta, e_{\beta} \rangle_\alpha = \int_{|z| \leq 1} |(e_\beta)^-_{\frac\alpha 2}(z)|^2  e^{-|z|^2} d\mu(z)
+ \int_{|z| \leq 1} |(e_{\beta} ) ^+_{\frac\alpha 2}(z)|^2  \frac{e^{-|z|^2}}{|z|^\alpha} \, d\mu(z). \end{equation*}

Thus, we have that \begin{align*} \sum_{\beta} & \int_{|z| \leq 1} |(e_{\beta} ) ^+ _{\frac\alpha 2}(z)|^2  \frac{e^{-|z|^2}}{|z|^\alpha} \, d\mu(z)  \\ & =  \sum_{|\beta| > \frac{\alpha}{2}} \frac{\Gamma (n + |\beta|)}{\beta ! \Gamma (n + |\beta| - \frac{\alpha}{2})} \int_{|z| \leq 1} |z ^\beta |^2  \frac{e^{-|z|^2}}{|z|^\alpha} \, d\mu(z)\\ & \leq \mu(B(0, 1)) \sum_{|\beta| > \frac{\alpha}{2}}  \frac{\Gamma (n + |\beta|)}{\beta! \Gamma (n + |\beta| - \frac{\alpha}{2})}  < \infty. \end{align*}

Similarly, we have that \begin{equation*} \sum_{\beta} \int_{|z| \leq 1} |(e_\beta)^-_{\frac\alpha 2} (z) |^2 e^{-|z|^2} d\mu(z) < \infty \end{equation*} which completes the proof.
\end{proof}

\begin{thm}\label{t:Schatten} Let $\mu$ be a nonnegative Borel measure on $\Cn$ and let $\alpha$ be a real number. Then the following are equivalent for every $0<p<\infty$.
\begin{itemize}
\item[(a)] $T^\alpha_\mu \in S^\alpha_p$.

\item[(b)] $\tilde\mu \in L^p (\Cn, dv)$.

\item[(c)] The function $ z \mapsto \mu(B(z,r)) $ is in
$L^p (\Cn, dv)$ for any $r > 0$.

\item[(d)] For any $r > 0$, let $\{a_j\} = s \mathbb{Z}^{2n}$ where $0 < s < r\sqrt{2/n}$ so that $\{B(a_j, r)\}$ covers $\Cn$. Then the sequence
$\{\mu (B(a_j, r))\} \in \ell^p$.

\end{itemize}

\end{thm}

Note that the equivalence between (c) and (d) can essentially be found in \cite{IZ}.  For convenience, we divide the rest of the proof of Theorem \ref{t:Schatten} into two parts.  In particular, we first finish the proof of the theorem when $1
\leq p < \infty$ and then finish the proof of the theorem when $0 < p < 1$. For the proof when $0 < p < 1$, we will need the following simple lemma that is well known in the classical Fock space setting (see \cite{Zhu}).

\begin{lem} \label{FrameOpBdd}  Let $s > 0$ and let $\{e_j\}$ be any orthonormal basis for $F_\alpha ^2$.  If $\{\xi_j\} \subset s \mathbb{Z}^{2n}$ and $A$ is the operator on $F _\alpha ^2$ defined by $A e_j := k_{\xi_j}$ then $A$ extends to a bounded operator on all of $F_\alpha ^2$ whose operator norm is bounded above by a constant that only depends on $s$.
\end{lem}

\begin{proof}  If $f, g \in F_\alpha ^2$, then the Cauchy-Schwarz inequality, Proposition \ref{weighteddiagonal}, Lemma \ref{mvplem}, and the reproducing property  gives us that \begin{align*} |\langle Af, g\rangle| & \leq \sum_{j = 1} ^\infty |\langle f, e_j\rangle_\alpha   \langle k_{\xi_j}, g\rangle_\alpha| \\ & \leq \|f\|_{F_\alpha ^2} \left(\sum_{j = 1} ^\infty  |\langle k_{\xi_j}, g\rangle_\alpha|^2\right)^\frac{1}{2} \\ & \approx  \|f\|_{F_\alpha ^2} \left(\sum_{j = 1} ^\infty |g(\xi_j)|^2 e^{-|\xi_j|^2} (1 + |\xi_j|)^{-\alpha} \right)^\frac{1}{2} \\ & \lesssim \|f\|_{F_\alpha ^2} \left(\sum_{j = 1} ^\infty \int_{B(\xi_j, r)} |g(u)|^2  e^{-|u| ^2 }  \, dv_\alpha (u)\right)^\frac{1}{2} \lesssim \|f\|_{F_\alpha ^2} \|g\|_{F_\alpha ^2} \end{align*} \end{proof}

\subsection{Proof of Theorem \ref{t:Schatten} for the case $ 1\leq p < \infty$}

(a) $\Rightarrow$ (b) : Note that a positive operator $T$ belongs to $S^\alpha_p$ if
and only if $T^p \in S^\alpha_1$. Let $T:=T^\alpha_\mu$. By Proposition 6.3.3 in
\cite{Zhu1},
$$\widetilde {T^p} (z) = \langle T^p k^\alpha_z, k^\alpha_z \rangle_\alpha \geq \langle T
k^\alpha_z, k^\alpha_z \rangle_\alpha^p = (\tilde\mu(z))^p.$$ Therefore, the
conclusion follows since
$$\infty > \mathrm{tr}(T^p) \approx \|\widetilde {T^p}\|_{L^p(\Cn,dv)} \geq
\|\tilde\mu\|_{L^p(\Cn,dv)}^p$$ by Lemma \ref{l:Trace}.

\medskip

(b) $\Rightarrow$ (c) : This is immediate by Lemma \ref{l:berezin}.

\medskip

(c) $\Rightarrow$ (a) :  We only prove this for $\alpha > 0$ since the case $\alpha \leq 0$ is similar. Furthermore, since condition (c) is independent of $r > 0$ (see \cite{IZ}), we can assume that $r < \frac{1}{2}$.  Let
$\{e_j\}$ be the standard orthonormal basis of normalized monomials for
$F^2_\alpha$. If $\varphi \in L^\infty(\Cn)$, and $T^\alpha_\varphi := T^\alpha_{\varphi\, dv}$, then $T^\alpha_\varphi$ is bounded on $F^2_\alpha$ and
\begin{align*} \langle T^\alpha_{|\varphi|} e_j, e_j \rangle_\alpha &=   \incn |(e_j) ^- _{\frac\alpha 2}(z)|^2  |\varphi(z)| e^{-|z|^2} dv(z)
\\ & + \incn|(e_j) ^+ _{\frac\alpha 2}(z)|^2  |\varphi(z)| \frac{e^{-|z|^2}}{|z|^\alpha} \, dv(z) \lesssim \|\varphi\|_{L^\infty(\Cn)}. \end{align*}
Therefore, $\|T^\alpha_{|\varphi|}\|_{S^\alpha_\infty} \leq
\|\varphi\|_{L^\infty(\Cn)}.$
Moreover, if $\varphi \in L^1(\Cn, dv)$, then
\begin{align*}
\sum_j |\langle T^\alpha_{|\varphi|} e_j, e_j \rangle_\alpha| &\leq \incn \sum_j |(e_j) ^- _{\frac\alpha 2}(z)|^2  |\varphi(z)| e^{-|z|^2} dv(z)
\\ & + \incn \sum_j |(e_j) ^+ _{\frac\alpha 2}(z)|^2  |\varphi(z)| \frac{e^{-|z|^2}}{|z|^\alpha} \, dv(z) \\
&\lesssim \incn |\varphi(z)|\, dv(z)
\end{align*}
by Proposition \ref{weighteddiagonal}.
Thus, by complex interpolation, we see that $T^\alpha_{|\varphi|} \in S^\alpha_p$ if
$\varphi \in L^p(\Cn,dv)$, and
\begin{equation}\label{e:Schatten lp} \| T^\alpha_{|\varphi|}\|_{S^\alpha_p}
\lesssim \|\varphi\|_{L^p (\Cn, dv)}
\end{equation}
for every $1\leq p\leq \infty$.

Now let $\varphi(z) := \mu(B(z,r)) \in L^p (\Cn, dv)$ and let $ C := \|\varphi\|_{L^p (\Cn, dv)}.$ Then for every $w\in \Cn$,
$$\int_{B(w, r/2)} (\varphi(z))^p \,dv(z) \leq C.$$ Since $\mu (B(w, r/2)) \leq \mu(B(z,r))$ for every $z\in B(w,
r/2)$, we see that
$$  \mu (B(w, r/2)) \lesssim  1$$ for every $w\in \Cn$.  Thus, Theorem \ref{t:Boundedness} gives us that the Toeplitz operators $T^\alpha_\varphi$ and $T^\alpha _\mu$ are bounded on
$F^2_\alpha$. By \eqref{e:Schatten lp} and Lemma \ref{l:trunToep}, it suffices to show that
\begin{align*} \langle T^\alpha_{\mu_2} f, f\rangle_\alpha \lesssim \langle T^\alpha_\varphi f,
f\rangle_\alpha\end{align*} for every polynomial $f$. However,
\begin{align*}
\langle T^\alpha_\varphi f, f\rangle_\alpha &=   \incn |f ^- _{\frac\alpha 2}(w)|^2     \varphi(w) e^{-|w|^2} dv(w)
\\ & + \incn|f ^+ _{\frac\alpha 2}(w)|^2  \varphi(w) \frac{e^{-|w|^2}}{|w|^\alpha} \, dv(w)
\\ & = \int_{\Cn} |f ^- _{\frac\alpha 2}(w)|^2  \mu(B(w, r)) e^{-|w|^2} dv(w)
\\ & + \int_{\Cn} |f ^+ _{\frac\alpha 2}(w)|^2  \mu(B(w, r)) \frac{e^{-|w|^2}}{|w|^\alpha} \, dv(w)
\\ & \geq \int_{|z| \geq 1 } \int_{B(z, r)} |f ^- _{\frac\alpha 2}(w)|^2   e^{-|w|^2} \, dv(w) \, d\mu(z)
\\ & + \int_{|z| \geq 1} \int_{B(z, r)} |f ^+ _{\frac\alpha 2}(w)|^2   \frac{e^{-|w|^2}}{|w|^\alpha} \, dv(w) \, d\mu(z)
\end{align*}
by Fubini's theorem.  Combining this with Lemma \ref{mvplem} and the fact that $r \leq \frac12$, we have that \begin{align*} \langle T^\alpha_\varphi f, f\rangle_\alpha  & \gtrsim \int_{|z| \geq 1} |f ^- _{\frac\alpha 2}(z)|^2   e^{-|z|^2} \, d\mu(z)
\\ & + \int_{|z| \geq 1} |f ^+ _{\frac\alpha 2}(z)|^2   \frac{e^{-|z|^2}}{|z|^\alpha}  \, d\mu(z)
\\ & = \langle T^\alpha _{\mu_2} f, f \rangle_\alpha. \end{align*}


\subsection{Proof of Theorem \ref{t:Schatten} for the case $0 < p \leq 1$}

(b) $\Rightarrow$ (c) : This follows immediately from Lemma \ref{l:berezin}.

\medskip

(d) $\Rightarrow$ (b) : Suppose that $\alpha > 0$.  We only show that \begin{equation*} \incn \left(\incn |(k^\alpha_z)_{\frac{\alpha}{2}} ^+ (w)|^2 \frac{e^{-|w|^2}}{|w|^\alpha} \, d\mu(w) \right)^p \, dv(z) < \infty \end{equation*} since the other term in the definition of $\widetilde{\mu}$ can be shown to be in $L^p(\Cn, dv)$ by a similar argument.  As in Lemma \ref{l:trunToep}, write $\mu = \mu_1 + \mu_2$ so that $\tilde \mu = \widetilde {\mu_1} + \widetilde{\mu_2} $.
For every $z, w\in\Cn$, we have that
$$(1+|z\cdot \overline w|)^\alpha \leq (1+|z||w|)^\alpha \leq (1+|z|)^\alpha
(1+|w|)^\alpha.$$ If $|w| \geq 1$ then this implies that
$$
|k^\alpha_z (w)|^2 \frac{e^{-|w|^2}}{|w|^\alpha} \lesssim  e^{-\frac{1}{4}|z-w|^2}
$$
by Lemma \ref{weightedmodified} and Proposition \ref{weighteddiagonal}. Therefore,
$$ \incn |(k^\alpha_z)_{\frac{\alpha}{2}} ^+ (w)|^2 \frac{e^{-|w|^2}}{|w|^\alpha} \, d{\mu_2} (w)   \lesssim \sum_j
\int_{B(a_j,r)}   e^{-\frac{1}{4}|z-w|^2} d\mu(w).$$   If $w\in
B(a_j,r)$, then $|z-w|^2 \geq  |z-a_j|^2
-2r|z-a_j|,$  so that
\begin{align*} \incn |(k^\alpha_z)_{\frac{\alpha}{2}} ^+ (w)|^2 \frac{e^{-|w|^2}}{|w|^\alpha} \, d{\mu_2} (w)   \lesssim \sum_j  \mu(B(a_j,r))
e^{-\frac{1}{4}|z-a_j|^2 + \frac{r}{2}|z-a_j|}.\end{align*}  Since $0<p\leq1$,
\begin{align*} \incn & \left(\incn |(k^\alpha_z)_{\frac{\alpha}{2}} ^+ (w)|^2 \frac{e^{-|w|^2}}{|w|^\alpha} \, d{\mu_2} (w)  \right)^p dv (z) \\ & \lesssim \sum_j  \mu(B(a_j,r))^p \incn
e^{-\frac{p}{4}|z-a_j|^2 + \frac{pr}{2}|z-a_j|}\,dv(z) < \infty. \end{align*}

Now we consider $\widetilde{\mu_1}$.  Note that if condition (d) is true, then one can easily show that $T_\mu ^\alpha$ is bounded.  Thus, we may assume that $|z| \geq 1$ since $\widetilde{\mu_1}$ is continuous. By Proposition \ref{weighteddiagonal} and the definition of $(k_z ^\alpha) ^+ _{\frac{\alpha}{2}}$ we have that \begin{align*}|(k^\alpha_z)_{\frac{\alpha}{2}} ^+ (w)| \frac{e^{-\frac12 |w|^2}}{|w|^\frac{\alpha}{ 2}}  & \leq  \frac{ e^{-\frac12 |z|^2}}{|z|^{\frac{\alpha}{2}} |w|^{\frac{\alpha}{2}}} \sum_{k > \frac{\alpha}{2}} \frac{ k^{\frac{\alpha}{2}} |z \cdot \overline{w}|^k}{k!}  \\ & \lesssim  \frac{ e^{- \frac12 |z|^2}}{|z|^{\frac{\alpha}{2}}} \sum_{k > \frac{\alpha}{2}} \frac{k^{\frac{\alpha}{2}} |z|^k}{k!} \approx |z|  e^{|z| - \frac12 |z|^2}   \end{align*}  if $|w| < 1$.  Thus, we have that \begin{align*} \int_{|z| \geq 1} & \left(\incn |(k^\alpha _z)_{\frac{\alpha}{2}} ^+ (w)|^2 \frac{e^{-|w|^2}}{|w|^\alpha} \, d{\mu_2} (w)  \right)^p dv (z) \\ & \leq \mu(B(0, 1)) \incn  |z| ^2 e^{2|z| - |z|^2} \, dv(z) < \infty. \end{align*}

Now let $\alpha \leq 0$.  Again, since $\tilde \mu$ is continuous, it suffices to show that $\tilde \mu \in L^p (\{|z| >1\}, dv)$. Now assume that $|z| >1$ and note that $(1 +
|z||w|)^\alpha \approx (1+|z|)^\alpha (1+|w|)^\alpha$ if $|w| >r$. Therefore, if
$|a_j| \geq 2r$ and $w\in B(a_j,r)$, then Lemma \ref{weightedmodified} and Proposition
\ref{weighteddiagonal} gives: $$|k^\alpha_z (w)|^2 e^{-|w|^2} \lesssim (1+|w|)^\alpha e^{-\frac{1}{4}|z-w|^2}.$$
If $|w| <3r$, then since $(1+|z||w|)^\alpha \leq 1$, we see that
$$|k^\alpha_z (w)|^2 e^{-|w|^2} \lesssim (1+|z|)^{|\alpha|}
e^{-\frac{1}{4}|z-w|^2}$$ by Lemma \ref{weightedmodified} and Proposition
\ref{weighteddiagonal} again. Thus, we have
\begin{eqnarray*}
\tilde\mu(z) &=& \incn |k^\alpha_z(w)|^2 |w| ^{|\alpha|} e^{-|w|^2} d\mu(w)\\
&\lesssim& \int_{|w|<3r}(1+|z|)^{|\alpha|} e^{-\frac{1}{4}|z-w|^2} d\mu(w)\\
& & +
\sum_{|a_j|\geq 2r} \int_{B(a_j,r)}   e^{-\frac{1}{4}|z-w|^2}
d\mu(w)\\
&\lesssim& (1+|z|)^{|\alpha|} e^{-\frac{1}{4}|z|^2 +\frac{3r}{2}|z| } \mu(B(0, 3r)
)\\
& &+ \sum_j  \mu(B(a_j,r)) e^{-\frac{1}{4}|z-a_j|^2 +
\frac{r}{2}|z-a_j|}.
\end{eqnarray*} Since the integrals
$$ \incn (1+|z|)^{|\alpha|p} e^{-\frac{p}{4}|z|^2 +\frac{3pr}{2}|z| } dv(z)$$ and
$$ \incn e^{-\frac{p}{4}|z-a_j|^2 + \frac{pr}{2}|z-a_j|}dv(z) =  \incn
e^{-\frac{p}{4}|z|^2 + \frac{pr}{2}|z|}dv(z)$$ are finite, $\tilde\mu \in L^p
(\{|z|>1\}, dv)$.

Now since (c) $\Leftrightarrow$ (d), we have that (b), (c), and (d) are equivalent.  To finish the proof, we will show that (b) $\Rightarrow$ (a) and (a) $\Rightarrow$ (d).
\medskip

(b) $\Rightarrow$ (a) : As we have already seen in (d) $\Rightarrow$ (b) for $1 \leq p < \infty$, the
operator $T:=T^\alpha_\mu$ is a positive bounded operator on $F^2_\alpha$. Note that
$T$ belongs to $S^\alpha_p$ if and only if $T^p$ is trace-class. Therefore,
$T$ belongs to $S^\alpha_p$ if $$\widetilde{T^p}(z) = \langle T^p k^\alpha_z,
k^\alpha_z \rangle_\alpha \in L^1(\Cn, dv)$$ by Lemma \ref{l:Trace}. Since
$$\langle T^p k^\alpha_z, k^\alpha_z \rangle_\alpha \leq \langle T k^\alpha_z,
k^\alpha_z \rangle_\alpha^p = (\tilde\mu (z))^p$$ for $0<p\leq 1$ (cf. Proposition 6.3.3
in \cite{Zhu1}), it follows that $T^\alpha_\mu \in S^\alpha_p$ if $\tilde \mu \in
L^p(\Cn, dv)$.

\medskip

(a) $\Rightarrow$ (d) : Assume without loss of generality that $r<1/2$. Since the
set $\{a_j : |a_j| \leq 2\}$ is finite, it suffices to show that the sequence
$$\{ \mu(B(a_j,r)) : |a_j| >2 \}$$ is in $\ell^p$. We will also only prove that (a) $\Rightarrow$ (d) when $\alpha > 0$ since the case $\alpha \leq 0$ is similar.  If $R >r $, then there exists a positive integer $N =N(R)$ and a partition $S_1
\cup\cdots \cup S_N$ of $\{a_j : |a_j| > 2\}$ such that $|a_i-a_j| >2R$ whenever
$i\neq j$ and $a_i, a_j \in S_k$ for some $k =1,...,N$. Let $S = \{b_j\}$ be one of the sets
$S_k$ and let
$$\nu := \sum_{j=1}^\infty \mu\,\chi_{B(b_j, r)}$$ where $\chi_{B(b_j, r)}$ is the
characteristic function on $B(b_j,r)$. Since $\nu \leq \mu$, $T^\alpha_\nu$ also
belongs to the Schatten class $S^\alpha_p$. Let $\{e_j\}$ be an orthonormal basis
for $F^2_\alpha$ and let $A$ be the bounded operator on $F^2_\alpha$ defined by
$A e_j = k^\alpha_{b_j}$ for $j=1,2,...$. Then the operator $T :=A^* T^\alpha_\nu
A \in S^\alpha_p$ and $\|T\|_{S^\alpha_p} \lesssim  \|T^\alpha_\nu\|_{S^\alpha_p}$ (by Lemma \ref{FrameOpBdd}). Let
$T = D + E$ where $D$ is the diagonal part of $T$ defined by
$$D f = \sum \langle T e_j, e_j \rangle_\alpha \langle f, e_j\rangle_\alpha e_j$$
and $E = T-D$. Note both $D$ and $E$ are compact. Since $0<p\leq 1$, we have that
that
\begin{equation}\label{e:TDE}\|T\|_{S^\alpha_p}^p \geq \|D\|_{S^\alpha_p}^p -
\|E\|_{S^\alpha_p}^p.
\end{equation}
From the definition of $D$ we obtain

\begin{align*}
\|D\|_{S^\alpha_p}^p &= \sum_{j=1}^\infty \langle Te_j, e_j\rangle_\alpha^p =
\sum_{j=1}^\infty \langle T^\alpha_\nu k^\alpha_{b_j},
k^\alpha_{b_j}\rangle_\alpha^p
\\ & = \sum_j \left(\incn |(k_{b_j} ^\alpha ) ^- _{\frac{\alpha}{2}} (w) | ^2 e^{- |w|^2} \,  d\mu(w)\right)^p
\\ & + \sum_j \left(\incn |(k_{b_j} ^\alpha ) ^+ _{\frac{\alpha}{2}} (w) | ^2 \frac{e^{- |w|^2}}{|w|^\alpha} \,  d\mu(w)\right)^p
\\ & \gtrsim \sum_j \left(\int_{B(b_j,r)} |k^\alpha _{b_j}(w)|^2 \frac{e^{-|w|^2}}{|w|^{\alpha}} \, d\nu(w)\right)^p
\end{align*} since $|w| \geq 1$.  Morever, since
$$|k^\alpha_{b_j}(w)|^2 = \frac{|K^\alpha(b_j,w)|^2}{|K^\alpha(b_j,b_j)|} \gtrsim
(1+|w|)^\alpha e^{|w|^2}$$ if $w\in B(b_j,r)$ by Propositions
\ref{weighteddiagonal} and
 \ref{t:lowerbound}, we conclude that
\begin{equation}\label{e:Estimates D}
\|D\|_{S^\alpha_p}^p \geq C_1 \sum_j \{ \mu (B(b_j,r))\}^p
\end{equation} for $r > 0$ small enough and some constant $C_1 >0$ independent of $R$.

By Lemma 6.36 in \cite{Zhu},
\begin{eqnarray*}
\|E\|_{S^\alpha_p}^p &\leq& \sum_{i=1}^\infty \sum_{j=1}^\infty |\langle E e_i,
e_j\rangle_\alpha|^p = \sum_{i\neq j} |\langle T^\alpha_\nu k^\alpha_{b_i},
k^\alpha_{b_j}\rangle_\alpha|^p \\
&=& \sum_{i\neq j} \left| \int_{|z| > 1}  (k^\alpha_{b_i})_{\frac{\alpha}{2}} ^- (z)\overline{(k^\alpha_{b_j})_{\frac{\alpha}{2}} ^- (z)}
e^{-|z|^2}  \, d\nu(z)\right|^p\\
& + & \sum_{i\neq j} \left| \int_{|z|>1} (k^\alpha_{b_i}) _{\frac{\alpha}{2}} ^+ (z)\overline{(k^\alpha_{b_j})_{\frac{\alpha}{2}} ^+  (z)}
\frac{e^{-|z|^2}}{|z|^{\alpha}} \, d\nu(z)\right|^p.
\end{eqnarray*} The last equality follows from the fact that $\nu =0$ on
$\{|z|<1\}$. Note
that if $|z| >1$ and $|w| >1$, then
\begin{equation}
\label{e:Upper estimates}
|K^\alpha (z,w)|
\lesssim (1+|z|)^{\frac\alpha 2}(1+|w|)^{\frac\alpha 2}
e^{\frac{1}{2}|z|^2 +\frac{1}{2}|w|^2 - \frac{1}{8}|z-w|^2}
\end{equation}
by Lemma \ref{weightedmodified}. Combining \eqref{e:Upper
estimates} with Proposition \ref{weighteddiagonal}, we see that
\begin{align*}
|(k^\alpha_{b_i})^+ _\frac{\alpha}{2} (z) {(k^\alpha_{b_j})^+ _\frac{\alpha}{2} (z)}| e^{-|z|^2}
\lesssim (1+|z|)^\alpha e^{ - \frac{1}{8}|z-b_i|^2 } e^{ - \frac{1}{8}|z-b_j|^2 }.
\end{align*}
Therefore,
\begin{align*}
& \|E\|_{S^\alpha_p}^p \\ &\lesssim \sum_{i\neq j} \left(  \incn  e^{
- \frac{1}{8}|z-b_i|^2} e^{ - \frac{1}{8}|z-b_j|^2}  d\nu(z)\right)^p\\
&= \sum_{i\neq j}\left(\sum_l \int_{B(b_l,r)} e^{ -
\frac{1}{8}|z-b_i|^2} e^{ - \frac{1}{8}|z-b_j|^2  } d\mu(z)\right)^p\\
&\leq \sum_{i\neq j}\sum_l \left(\int_{B(b_l,r)}   e^{ -
\frac{1}{8}|z-b_i|^2} e^{ - \frac{1}{8}|z-b_j|^2}  d\mu(z)\right)^p \label{sum}
\end{align*} since $0<p\leq 1$.

 Moreover, since $i\neq j$, we have that $|b_i-b_j| >2R$. Therefore, either $|z-b_i|
>R$ or $|z-b_j| >R$ for every $z\in \Cn$, which implies that
$$e^{ - \frac{1}{8}|z-b_i|^2  - \frac{1}{8}|z-b_j|^2} \leq e^{-\frac{1}{16}R^2} e^{
- \frac{1}{16}|z-b_i|^2  - \frac{1}{16}|z-b_j|^2}.$$  Therefore,
$$\|E\|_{S^\alpha_p}^p \lesssim e^{-\frac{pR^2}{16}} \sum_{i\neq j}\sum_l  \left(\int_{B(b_l,r)} e^{ - \frac{1}{16}|z-b_i|^2  - \frac{1}{16}|z-b_j|^2} d\mu(z)\right)^p.$$ For $z\in B(b_l,r)$, if $l\neq i$, then
$$|z -b_i| \geq |b_l - b_i | -r \geq \frac{1}{2} |b_l -b_i|,$$ since $r <R <|b_l
-b_i|/2$. If $l=i$, then the inequality above holds obviously. Therefore,
\begin{eqnarray*}
\|E\|_{S^\alpha_p}^p &\lesssim& e^{-\frac{pR^2}{16}} \sum_l \{
\mu (B(b_l,r))\}^p \sum_{i\neq j} e^{ - \frac{p}{64}|b_l-b_i|^2  -
\frac{p}{64}|b_l-b_j|^2}\\
&\leq& e^{-\frac{pR^2}{16}} \sum_l \{\mu (B(b_l,r))\}^p
\left(\sum_j  e^{ - \frac{p}{64}|b_l-b_j|^2 }\right)^2\\
&\leq& e^{-\frac{pR^2}{16}} \sum_l \{ \mu (B(b_l,r))\}^p
\left(\sum_j  e^{ - \frac{p}{64}|b_l-a_j|^2 }\right)^2\\
&=& e^{-\frac{pR^2}{16}} \sum_l \{\mu (B(b_l,r))\}^p
\left(\sum_j  e^{ - \frac{p}{64}|a_j|^2 }\right)^2.
\end{eqnarray*}
Since the last series converges, we conclude that
\begin{equation}\label{e:Estimates E}
\|E\|_{S^\alpha_p}^p \leq C_2 e^{-\frac{pR^2}{16}} \sum_l  (\mu
(B(b_l,r)))^p
\end{equation} for some constant $C_2 < \infty$ independent of $R$.

 Therefore, if $R$ if chosen where $C_1 - C_2 e^{-\frac{pR^2}{16}}  >0$, then it follows that
$$\sum_j  \{\mu (B(b_j,r))\}^p \lesssim \|T_\mu ^\alpha\|_{S_p ^\alpha} ^p $$ from \eqref{e:TDE},
\eqref{e:Estimates D} and \eqref{e:Estimates E} so that
$$\sum_{j : |a_j| > 2}  \{\mu (B(a_j,r))\}^p \lesssim \|T_\mu ^\alpha\|_{S_p ^\alpha} ^p$$ for all $\mu$ where $\sum_{j : |a_j| > 2}  \{\mu (B(a_j,r))\}^p < \infty$ since $N$ only depends on $R$.  However, an easy approximation argument then gives us that $\{\mu (B(a_j,r))\} \in \ell^p$ for all $\mu$ such that $T_\mu ^\alpha \in S_p ^\alpha$.

\section{Remarks and Open Problems} \label{SectionRAOP}

As was remarked in the introduction, it would be interesting to know if one can prove Theorems A, B, and C for Toeplitz operators with nonnegative Borel measure symbols that are defined in terms of the reproducing kernel of $F_\alpha ^2$ with respect to the inner product \eqref{CanInnProd}. Note that as usual, the main difficulty in doing this is obtaining the necessary estimates for the reproducing kernel defined with respect to the inner product \eqref{CanInnProd}. Furthermore, it would be quite interesting to know if there is a direct operator theoretic argument that allows one to directly transfer information about Toeplitz operators defined with respect to one Hilbert space inner product to a Toeplitz operator defined with respect to a different Hilbert space inner product (where the norms these inner products generate are equivalent.) In particular, it would be interesting to know if one can employ operator theoretic techniques to prove Theorems A, B, or C (or even portions of these theorems) for Toeplitz operators defined with respect to $\langle \cdot, \cdot \rangle_{L_\alpha ^2}$ directly from Theorems A, B, or C, respectively.

We should remark that Theorems A,  B, and C are known to be true for Toeplitz operators with nonnegative Borel measure symbols on a very wide class of generalized Fock spaces.  More precisely, let $d^c =  \frac{i}{4} (\overline{\partial} - \partial)$ and let $d$ be the usual exterior derivative. Let $\phi \in C^2(\Cn)$ be a real valued function on $\Cn$  such that \begin{equation} c \omega_0 < d d^c \phi < C \omega_0 \label{PhiCond} \end{equation} holds uniformly pointwise on $\Cn$ for some positive constants $c$ and $C$ (in the sense of positive $(1, 1)$ forms) where $\omega_0 = d d^c |\cdot |^2$ is the standard Euclidean K\"{a}hler form.  For $0 < p < \infty$, let $L_\phi ^p$ be the space of measurable $h$ on $\Cn$ where \begin{equation*} \|h\|_{L_\phi ^p} := \incn \left|h(z) e^{- \phi(z)} \right|^p \, dv(z) < \infty \end{equation*} and let $F_\phi ^p$ be the so called ``generalized Fock space" introduced in \cite{SV} and defined by \begin{equation*} F_\phi ^p  := \{ f \in H(\Cn) : \|f\|_{L_\phi ^p} < \infty \}. \end{equation*} If $K(z, w)$ is the reproducing kernel of $F_\phi ^2$ (with respect to the canonical inner product) and $\mu$ is a positive Borel measure on $\Cn$, then define the Toeplitz operator $T_\mu ^\phi$ by \begin{equation*} T_\mu ^\phi f (z) := \incn K(z, w) f(w) e^{-2\phi(w)}\, dv(w). \end{equation*}

\noindent Then the following was proved in \cite{IVW, SV}:

\begin{thm} \label{SVthm} If $\phi \in C^2(\Cn)$ satisfies condition (\ref{PhiCond}) and $\mu \geq 0$, then the following are equivalent:
\begin{itemize}
\item[(a)] The Toeplitz operator $T^\phi _\mu$ extends to a bounded operator on
$F^p _\phi$.
\item[(b)] The Berezin transform $\tilde\mu$ is bounded on $\Cn$.
\item[(c)] The function $z \mapsto \mu (B(z,r)) $ is bounded for
every $r > 0$.
\item[(d)] $\mu$ is a Carleson measure for $F_\phi ^p$ for any $1 \leq  p < \infty$.
\end{itemize}

Furthermore, the natural corresponding ``little oh'' conditions characterize compact Toeplitz operators and vanishing Carleson measures on $F_\phi ^p$ for $1 \leq p < \infty$, and the natural corresponding $\ell^p$ and $L^p$ conditions characterize Schatten $p$ class Toeplitz operators on $F_\phi ^2$ for $0 < p < \infty$.
\end{thm}

Note that our weighted Fock spaces $F_\alpha ^p$ clearly do not fall fall under the class of weighted Fock spaces dealt with in Theorem \ref{SVthm} since the factor $(1 + |z|)^{-\alpha}$ in the definition of $F_\alpha ^p$ is \textit{not} being raised to the $p$ power.  However, note that the weighted Fock space $F_{p \alpha} ^p$ for any $\alpha \in \mathbb{R}$ \textit{does} in fact fall under these classes of weighted Fock spaces. In particular, by a standard closed graph theorem argument, we have that $f \in F_{p \alpha} ^p$ if and only if $z \mapsto (A + |z|^2 )^\frac{\alpha}{2} e^{-\frac{1}{2} |z|^2} f(z)  $ is in $L^p(\Cn, dv)$ for any $A > 0$, and furthermore the canonical norm induced by this condition (for fixed $A > 0$) is equivalent to the $F_{p \alpha} ^p$ norm.   Thus, if \begin{equation*} \phi(z) := \frac{1}{2} |z|^2 - \frac{\alpha}{2} \ln (A + |z|^2)  \end{equation*} then a relatively straight forward computation (see \cite{I} for details)  tell us that $\phi$ satisfies (\ref{PhiCond}) if $A > 2\alpha$.

However, treating $F_{p \alpha } ^p$ as a generalized Fock space, we obviously have that the reproducing kernel is completely different than the one in Theorems A, B, and C, and thus the Berezin transform in the statement of Theorem \ref{SVthm} for $F_{p \alpha } ^p$ is considerably different than the one in the statement of Theorems A,  B, and C.  In other words, Theorem \ref{SVthm} does \textit{not} imply Theorems A, B, and C even for the weighted Fock spaces $F_{p \alpha } ^p$.  Taking this into account, it is interesting that the exact same conditions in both Theorems A, B, C, and Theorem \ref{SVthm} characterize the boundedness, compactness, and Schatten class membership of Toeplitz operators.

\section*{Acknowledgements}The authors would like to thank Brett Wick for bringing to our attention the paper \cite{SV}.

\end{document}